\documentclass[11pt]{amsart}
\usepackage[active]{srcltx}
\usepackage{xifthen}
\usepackage[all]{xypic}
\usepackage{tikz}
\usetikzlibrary{patterns}
\usepackage{hyperref}
\usepackage{enumitem}
\setlist[enumerate]{label = $(\alph*)$, leftmargin = *}
\usepackage{amssymb}
\usepackage{amsmath}
\usepackage{amsthm}
\usepackage[mathcal]{euscript}
\numberwithin{equation}{section}

\newtheorem{thm}{Theorem}[section]
\newtheorem{cor}[thm]{Corollary}
\newtheorem{lem}[thm]{Lemma}
\newtheorem{prop}[thm]{Proposition}
\theoremstyle{definition}
\newtheorem{deff}[thm]{Definition}
\newtheorem{rem}[thm]{Remark}

\newsavebox\ideabox
\newenvironment{idea}
{\begin{equation}\tag{Q}
    \begin{lrbox}{\ideabox}
      \begin{minipage}{\dimexpr\columnwidth-2\leftmargini}
        \setlength{\leftmargini}{0pt}%
        \begin{quote}}
        {\end{quote}
      \end{minipage}
    \end{lrbox}\makebox[0pt]{\usebox{\ideabox}}
  \end{equation}}
\makeatletter
\def\namedlabel#1#2{\begingroup
   \def\@currentlabel{#2}%
   \label{#1}\endgroup
 }
 
\newcommand{\DRH}{{\sf DRH}}
\newcommand{\h}{{\sf H}}
\newcommand{\W}{{\sf W}}
\newcommand{\V}{{\sf V}}
\newcommand{\A}{{\sf A}}
\newcommand{\G}{{\sf G}}
\newcommand{\R}{{\sf R}}
\newcommand{\Sl}{{{\sf Sl}}}
\newcommand{\s}{{{\sf S}}}
\newcommand{\Ab}{{{\sf Ab}}}
\newcommand{\DO}{{{\sf DO}}}

\newcommand{\pseudo}[2]{\overline{\Omega}_{#1} {\sf{#2}}}
\newcommand{\pseudosig}[2]{\Omega^{\sigma}_{#1} {\sf{#2}}}

\newcommand{\lbf}[2][]{{\sf{lbf}}\ifthenelse{\isempty{#1}}{}{_{#1}}(#2)}
\newcommand{\cum}[1]{\vec{c}(#1)}
\newcommand{\card}[1]{\left\lvert#1\right\rvert}
\newcommand\malcev{\mathop{\raise0.5pt\hbox{\footnotesize$\bigcirc$\kern-8.5pt\raise0.5pt\hbox{\tiny$m$}\kern2pt}}}

\newcommand{\Req}{\mathrel{\mathcal R}}
\newcommand{\Deq}{\mathrel{\mathcal D}}
\newcommand{\Heq}{\mathrel{\mathcal H}}

\newcommand{\iC}{\mathcal{C}}
\newcommand{\iS}{\mathcal{S}}

\begin{document}
\title[Reducibility properties for $\DRH$]{Some reducibility properties for pseudovarieties of the
  form~$\DRH$}
\author{C\'elia Borlido}
\address{Centro de Matem\'atica e Departamento de Matem\'atica, Faculdade de Ci\^encias,
Universidade do Porto, Rua do Campo Alegre, 687, 4169-007 Porto,
Portugal}
\email{cborlido@fc.up.pt}
\begin{abstract}
  Let $\h$ be a pseudovariety of groups and $\DRH$ be the pseudovariety
  containing all finite semigroups whose regular $\Req$-classes belong to $\h$.
  We study the relationship between reducibility of~$\h$ and of $\DRH$ with
  respect to several particular classes of systems of equations.
  The classes of systems considered (of pointlike, idempotent pointlike
  and graph equations) are known to play a role in decidability
  questions concerning pseudovarieties of the forms $\V * \W$, $\V
  \vee \W$, and~$\V \malcev \W$.
\end{abstract}
\maketitle
\section{Introduction}
The interest in studying pseudovarieties of semigroups is, in part,
justified by Eilenberg's
correspondence~\cite{MR0530383}, which establishes a
bijection between pseudovarieties of
finite semigroups and varieties of rational languages.
Also, rational languages are a very important object in
Theoretical Computer Science, as they correspond to the languages
recognized by finite state automata.

In turn, pseudovarieties are quite often described as a result of
applying certain operators on pairs of other pseudovarieties, such as the
semidirect product $*$, the join $\vee$, and the Mal'cev product
$\malcev$.
Therefore, it is a natural question to ask whether pseudovarieties of
the form $\V*\W$, $\V\vee\W$, or $\V \malcev \W$ are
\emph{decidable} (meaning that they have a decidable \emph{membership
  problem}).
It is known that $\V$ and $\W$ being decidable is not enough
to have decidability of any of those pseudovarieties~\cite{MR1150933,MR1723477}.
It was the search for sufficient conditions to preserve decidability
under the operator $*$ that led to the definition of
\emph{hyperdecidability}, a stronger notion of decidability~\cite{hyperdecidable}.
Shortly after, the notion of \emph{tameness}~\cite{steinbergJA,MR1750489}
emerged as a method of establishing hyperdecidability of
pseudovarieties.
Briefly, it may be described in two steps:
\emph{decidability of the word problem} and \emph{reducibility}.
Some other variants of stronger versions of decidability may be
found in the literature~(see~\cite{unified_theory} for an overview).

It is also worth mentioning that a particular instance of
hyperdecidability, known as \emph{strong
decidability}, was already considered for several years
under the name of \emph{computable pointlike sets}.
For instance, in 1988 Henckell~\cite{MR968571}  proved that aperiodic
semigroups have
computable pointlike sets or, in other words, that the pseudovariety
$\A$ of aperiodic finite semigroups is strongly decidable.
This study was conducted to produce progress in the question of
decidability of the Krohn-Rhodes complexity for semigroups~\cite{MR0236294}.
Along the same line, Ash~\cite{MR1232670} introduced \emph{inevitable sequences}
in a finite
monoid (for finite groups) in order to prove the Rhodes type II
conjecture~\cite{MR645703}.
Deciding whether a sequence $(s_1, \ldots, s_n)$ from a finite
monoid is inevitable in Ash's sense translates to hyperdecidability of the
pseudovariety~$\G$ of finite groups with respect to the
equation~$x_1\cdots x_n = 1$.
Also, Pin
and Weil~\cite{pinWEIL} described a defining set of identities for a
Mal'cev product, which in turn implies that the decidability of idempotent
pointlike sets may be used as a sufficient condition for decidability of
Mal'cev products of pseudovarieties.
The diversity of motivations behind these works somehow indicates that
hyperdecidability may lead the way to a better understanding of
the structure of finite semigroups.
Indeed, many researchers have shown interest in studying strong
versions of decidability for pseudovarieties (see, for instance,
\cite{MR1859280,JA,reducibility,MR2142087,MR1611659,MR1112302,2007arXiv0706.0248H,MR1818661}).

On the other hand, the pseudovarieties of the form~$\DRH$ have already
been considered in the literature.
In the mid seventies, Sch\"utzenberger~\cite{MR0444824} identified the associated
varieties of rational languages under Eilenberg's correspondence.
Also, more recently, a study on the structure of the free pro-$\DRH$
semigroup was carried out by Almeida and Weil \cite{drh}.
Pseudovarieties of the form~$\DRH$ are the object of our study, in which we 
answer the following question:

\begin{idea}\namedlabel{q1}{Q}
  Given an implicit signature~$\sigma$, what conditions on a
  pseudovariety of groups $\h$
  guarantee that the pseudovariety $\DRH$ is \emph{$\sigma$-reducible with
  respect to a given class~$\iC$ of finite systems of equations} (to be
  precisely described in Subsection~\ref{sec:1})?
\end{idea}

\noindent The classes~$\iC$ considered are precisely those related with
the decidability problems mentioned above. More precisely, we consider
systems of \emph{pointlike equations} ($x_1 = \cdots = x_n$), of \emph{graph
  equations} (equations arising from finite graphs by assigning to
each edge  $x \xrightarrow y z$ the
equation $xy = z$), and of \emph{idempotent
pointlike equations} ($x_1 = \cdots = x_n = x_n^2$).

The paper is organized as follows.
We devote Section \ref{sec:1002} to an overview of
results in the literature that we use in the rest of the
paper.
In particular, in Subsection~\ref{sec:1} we expose
some concepts and results concerning decidability.
The subsequent sections focus on pointlike, graph, and idempotent
pointlike equations, in this order.
We prove in Section \ref{sec:1003} that $\h$ being
$\sigma$-reducible with respect to systems of pointlike equations,
suffices for $\DRH$ to enjoy the same property.
That result is achieved by considering a certain periodicity
phenomenon on the constraints.
Then, in Section \ref{sec:1004}, we study systems of graph equations.
We prove that $\h$ is $\sigma$-reducible with respect to systems of
graph equations if and only if so is $\DRH$.
For that purpose, we borrow from~\cite{Rtame} the notion of \emph{splitting
  point} considered in the setting of the pseudovariety $\R$.
Finally, in the last section, we prove that if $\h$ is
$\sigma$-reducible with respect to systems of graph equations, then
$\DRH$ is $\sigma$-reducible with respect to systems of idempotent
pointlike equations.
The techniques used are somehow similar to the ones used in Section
\ref{sec:1002}.

\section{Preliminaries}\label{sec:1002}

We assume that the reader is familiar with the theory of finite and profinite
semigroups. We refer to~\cite{livro,profinite} for this topic.
For the basics concepts and results on topology, the reader is
referred to~\cite{topologia}.
\subsection{General definitions and notation}

In the sequel, $\V$ and $\W$ stand for arbitrary pseudovarieties of
semigroups, while $\h$ stands for an arbitrary pseudovariety of groups.
We list below the pseudovarieties mentioned in this paper.
\begin{itemize}[label =]
\item $\s$ consists of all finite semigroups;
\item $\Sl$ consists of all finite semilattices;
\item $\G$ consists of all finite groups;
\item $\Ab$ consists of all finite Abelian groups;
\item $\G_p$ consists of all finite $p$-groups (for a prime number $p$);
\item $\G_{sol}$ consists of all finite solvable groups;
\item $\R$ consists of all finite $\Req$-trivial semigroups;
\item $\DRH$ consists of all finite semigroups whose regular
  $\Req$-classes are groups of $\h$;
\item $\DO$ consists of all finite semigroups whose regular
  $\Deq$-classes are orthodox semigroups;
\item $\overline \h$ consists of all finite semigroups whose subgroups
  belong to $\h$.
\end{itemize}

Let $A$ be a finite alphabet. The free $A$-generated pro-$\V$ 
semigroup is denoted $\pseudo AV$.
Whenever $\V$ is not the trivial pseudovariety, it is usual to
identify $A$ with its image under the generating
mapping of $\pseudo AV$, so that the free semigroup $A^+$ is a
subsemigroup of~$\pseudo AV$.
For a subpseudovariety~$\W$ of~$\V$, we represent by~$\rho_\W$
the canonical projection from~$\V$ onto~$\W$, should~$\V$ be clear
from the context.
When $\Sl \subseteq \V$, we denote $\rho_\Sl$ by $c$ and call it the
\emph{content function}.
An \emph{implicit signature} is a set of pseudowords generically
denoted~$\sigma$.
Each pseudoword may be naturally seen as an implicit
operation~\cite[Theorem~4.2]{profinite}.
Hence, each profinite semigroup is endowed with a structure of
$\sigma$-algebra.
We denote by $\pseudosig AV$ the free $A$-generated semigroup
over~$\V$.
Further, we let $\langle\sigma\rangle$ denote the implicit signature
obtained from $\sigma$ through composition of its elements (see
\cite[Proposition~4.7]{profinite}).
The implicit operations corresponding to the elements of $A^+$ are called
\emph{explicit operations}.
The $\omega$-power is the implicit operation~$x^\omega$ that assigns to each
element~$s$ of a finite semigroup the unique idempotent that is a
power of~$s$.
It plays a distinguished role in this paper.
We call \emph{pseudowords over $\V$} (or simply pseudowords, when~$\V =
\s$) the elements of $\pseudo AV$, and \emph{$\sigma$-words over $\V$}
(or simply $\sigma$-words, when~$\V =
\s$) the
elements of $\pseudosig AV$.

If $S$ is a semigroup, then we represent by $S^I$ the monoid with
subsemigroup~$S$, identity~$I$, and underlying set given by $S \uplus \{I\}$.
Based on the identification $A^+ \subseteq \pseudo AV$, we sometimes
call \emph{empty word} the identity element $I \in (\pseudo AV)^I$. We further set $c(I)
= \emptyset$.

Given a formal equality of pseudowords $u = v$, also called
\emph{pseudoidentity}, we write $u =_\V v$ if the interpretations of
$u$ and $v$ coincide on every semigroup of~$\V$. Note that this is
equivalent to having $\rho_\V(u) = \rho_\V(v)$.
All the expressions \emph{$u = v$ modulo~$\V$}, \emph{$\V$ satisfies
  $u = v$}, and \emph{$u = v$ holds in~$\V$} mean that $u =_\V v$.
\subsection{The pseudovariety $\DRH$}

For a complete study of pseudovarieties of the form~$\DRH$, the reader
is referred to~\cite{drh}.
We proceed with the statement of some structural properties of the
free pro-$\DRH$ semigroup that we use later.

It is well known that for every element $u$ of $\pseudo AS$
(respectively, of $\pseudo A{DRH}$) there exists a unique
factorization $u = u_\ell a u_r$, with $u_\ell$ and $u_r$ possibly the
empty word, such that 
$c(u_\ell a)= c(u)$ and $a \notin c(u_\ell)$ (see, for instance, \cite[Proposition
2.1]{palavra} and \cite[Proposition 2.3.1]{drh}).
Such a factorization (both over $\s$ and over $\DRH$) is called the
\emph{left basic factorization of $u$}.

Let $u$ be either a pseudoword or a pseudoword over $\DRH$.
For each $k \ge 1$, we define $\lbf[k]u$ inductively as follows.
If $u = u_{1,\ell} a_1 u_{1,r}$ is the left basic factorization of $u$, then we
set $\lbf[1]u = u_\ell$. For $k > 1$, we set $\lbf[k] u = I$ if
$\lbf[k-1] u = I$, and we set $\lbf[k]u = u_{k-1,\ell}$ if the left
basic factorization of $\lbf[k-1]u$ is given by $\lbf[k-1]u = u_{k-1,
  \ell}a_{k-1}u_{k-1,r}$.
The \emph{cumulative content of $u$}, denoted $\cum u$, is the
ultimate value of the
sequence $(c(\lbf[k]u))_{k \ge 1}$. Observe that this sequence indeed
stabilizes since it forms a descending chain of subsets of some finite
set~$A$.

On the other hand, if we consider the iteration of the left basic
factorization to the leftmost factor, then we obtain uniqueness of the
so-called \emph{first-occurrences factorization}.
We state that fact for later reference.

\begin{lem}\label{c:5}
  Let $u$ be a pseudoword (respectively, a pseudoword
  over~$\DRH$).
  Then, there exists a unique factorization
  $u = a_1u_1 a_2u_2\cdots  a_nu_n$
  over $\s$ (respectively, over $\DRH$) such that $a_i \notin c(a_1u_1
  \cdots a_{i-1}u_{i-1})$, for  $i = 2, \ldots, n$, and $c(u) = \{a_1,
  \ldots, a_n\}$.
\end{lem}
We say that $ua$ is an \emph{end-marked pseudoword} provided $a \notin
\cum u$.
Also, the product~$uv$ is \emph{reduced} if~$v$ is nonempty and the first
letter of~$v$ (which is defined, by Lemma~\ref{c:5}) does not belong
to the cumulative content of~$u$.
The following result is used later.
\begin{prop}
  [{\cite[Proposition 4.8]{JA}}]
  \label{p:2}
  The set of all end-marked pseudowords over a finite alphabet
  constitutes a well-founded forest under the partial order
  $\le_{\Req}$.
\end{prop}

We end this subsection with some results concerning identities over
$\DRH$.
They seem to be already used in the
literature, however, since
we could not find the exact statement that fits our purpose, we
include the proofs for the sake of completeness.

\begin{lem}
  \label{sec:13}
  Let $u, v$ be pseudowords. Then, $\rho_{\DRH} (u)$ and
  $\rho_\DRH(v)$ lie in the same $\Req$-class if and only if the
  pseudovariety $\DRH$ satisfies $\lbf[k] u = \lbf[k] v$, for every $k
  \ge 1$.
\end{lem}
\begin{proof}
  The sufficient condition follows straightforwardly from the
  definitions of the relation~$\Req$, of
  left basic factorization, and of cumulative content.
  Conversely, suppose that~$\DRH$ satisfies $\lbf[k] u = \lbf[k] v$, for every $k
  \ge 1$. Then, we may choose accumulation points of the sequences
  $(\lbf[1]u\cdots \lbf[k]u)_{k \ge 1}$ and $(\lbf[1]v\cdots
  \lbf[k]v)_{k \ge 1}$, say $u_0$ and $v_0$ respectively, such that
  $\DRH$ satisfies $u_0 = v_0$.
  Since $u_0$ and $v_0$ are both $\Req$-below $u$ and $v$
  modulo~$\DRH$, the result follows from a simple computation.
\end{proof}

The next result may be considered the key ingredient for the
representations of elements of $\pseudo A{DRH}$ presented in
\cite{drh}, of which we implicitly make use.

\begin{prop}
  [{\cite[Proposition 5.1.2]{drh}}]
  \label{p:3}
  Let $\V$ be a pseudovariety such that the inclusions $\h\subseteq \V \subseteq \DO
  \cap \overline \h$ hold. Then, the regular $\Heq$-classes of $\pseudo AV$
  are free pro-$\h$ groups on their content. More precisely, if $e$ is
  an idempotent of $\pseudo AV$ and if $H_e$ is its $\Heq$-class, then
  letting $\psi_e(a) = eae$ for each $a \in c(e)$ defines a unique
  homeomorphism $\psi_e:\pseudo{c(e)}H \to H_e$ whose inverse is the
  restriction of $\rho_{\V, \h}$ to $H_e$.
\end{prop}

Before proving the last result of this subsection, we need to
introduce a definition.
Let~$u$ be a pseudoword such that~$c(u) = \cum u$.
Then, all accumulation points of the sequence
\begin{equation}
  \label{eq:10000}
  (\rho_\DRH(\lbf[1]u)\cdots
\rho_\DRH(\lbf[k]u))_{k \ge 1}
\end{equation}
 lie in the same regular $\Req$-class
\cite[Proposition~2.1.4]{drh}, which in turn, by
definition of~$\DRH$, is a group.
The identity of that group is said to be the \emph{idempotent
  designated by the sequence~\eqref{eq:10000}}.

The following lemma becomes trivial in the particular case of $\DRH = \R$.

\begin{lem}
  \label{sec:16}
  Let $u,v \in \pseudo AS$ and $u_0, v_0 \in (\pseudo AS)^I$ be such
  that $c(u_0) \subseteq \cum u$ and $c(v_0) \subseteq \cum
  v$. Then, the pseudovariety $\DRH$ satisfies $uu_0 = vv_0$ if and
  only if it satisfies $u \Req v$ and if, in addition, the pseudovariety
  $\h$ satisfies $uu_0 = vv_0$. In particular, by taking $u_0 = I =
  v_0$, we get that $u =_\DRH v$ if and only if $u \Req v$
  modulo~$\DRH$ and $u =_\h v$.
\end{lem}
\begin{proof}
  Let $u$, $v$, $u_0$, and $v_0$ be pseudowords satisfying the
  hypothesis of the lemma. We start by noticing that the inclusions
  $c(u_0) \subseteq \cum u$ and $c(v_0) \subseteq \cum v$ imply,
  respectively, that
  $\lbf[k] u = \lbf[k]{uu_0}$ and $\lbf[k] v = \lbf[k]{vv_0}$, for
  every $k \ge 1$.
  
  Let us suppose that $\DRH$ satisfies $uu_0 = vv_0$.
  Then, for every $k \ge 1$, we have $\lbf[k]u = \lbf[k]{uu_0} =
  \lbf[k]{vv_0} = \lbf[k]v$. 
  By Lemma~\ref{sec:13}, it follows that $u$ and $v$ are $\Req$-equivalent
  modulo $\DRH$.
  The pseudoidentity $uu_0 =_\h vv_0$ follows from the fact
  that $\h$ is a subpseudovariety of $\DRH$.

  Conversely, we assume that $u \Req v$ modulo $\DRH$ and that $uu_0 =
  vv_0$ modulo $\h$. Invoking Lemma~\ref{sec:13} again, we have
  $\lbf[k] u =_\DRH \lbf[k]v$, for every $k \ge 1$.
  If $\cum u = \cum v = \emptyset$, then $u_0 = v_0 = I$ and the
  pseudoidentity $uu_0 =_\DRH vv_0$ is immediate.
  Otherwise, if $\cum u = \cum v \neq \emptyset$, then we let $m$ be
  such that $c(\lbf[m]u) = \cum u$ (and consequently, $c(\lbf[m]v) =
  \cum v$), and we let $u_1 = \lbf[1]u\cdots \lbf[m-1]u$, $v_1 =
  \lbf[1]v\cdots \lbf[m-1]v$, and $u_2, v_2$ be the unique pseudowords
  such that $u = u_1u_2$ and $v = v_1v_2$.
  Note that  $\lbf[k]{u_2} = \lbf[m+k-1]u$, for
  every~$k \ge 1$ and hence,  the equalities $c(u_2) = \cum{u_2} =
  \cum u$ hold (similarly for~$v_2$).
  Since $\lbf[k] u =_\DRH \lbf[k]v$ ($k \ge 1$),
  the idempotent
  designated by the sequence $(\rho_\DRH(\lbf[m]u) \cdots
  \rho_\DRH(\lbf[k]u))_{k \ge m}$
  is the same as the idempotent designated by the sequence
  $$(\rho_\DRH(\lbf[m]v) \cdots \rho_\DRH(\lbf[k]v))_{k \ge m},$$
  say~$e$.
  On the other hand, since $\rho_\h(u_2u_0) = \rho_\h(v_2v_0)$, it
  follows that $\psi_e(\rho_\h(u_2u_0)) = \psi_e(\rho_\h(v_2v_0)) $, where
  $\psi_e$ is the homeomorphism of Proposition~\ref{p:3}.
  Finally, since both $\rho_\DRH(u_2u_0)$ and $\rho_\DRH(v_2v_0)$ lie
  in the $\Heq$-class of~$e$, Proposition~\ref{p:3} yields that they are
  equal (because $\rho_\h$ is the inverse of~$\psi_e$).
\end{proof}

\subsection{Decidability}\label{sec:1}

The \emph{membership problem} for a pseudovariety $\V$ amounts to
determining whether a given finite semigroup belongs to $\V$.
If there exists an algorithm to solve this problem, then the
pseudovariety $\V$ is said to be \emph{decidable}.
As we already referred in  the Introduction,
other stronger notions of decidability have been set up over the
years. They are related with so-called systems of pseudoequations.

Let $X$ be a finite set of \emph{variables}. A \emph{pseudoequation}
is a
formal expression $u = v$ with $u, v \in \pseudo {X} S$. If
$u,v \in \pseudosig {X}S$, then $u = v$ is said to be a
\emph{$\sigma$-equation}. A \emph{finite system of pseudoequations}
(respectively, \emph{$\sigma$-equations}) is a
finite set
\begin{equation}
  \label{eq:35}
  \{u_i = v_i \colon i = 1, \ldots, n\},
\end{equation}
where each $u_i = v_i$ is a pseudoequation (respectively,
$\sigma$-equation). For each
variable $x \in X$, we consider a \emph{constraint} given by a
clopen subset $K_x$ of $\pseudo AS$. 
Then, a
\emph{solution modulo~$\V$}
of the system \eqref{eq:35}
\emph{satisfying the given constraints} is a continuous homomorphism $\delta:
\pseudo {X} S \to \pseudo AS$ such that the following conditions are
satisfied:
\begin{enumerate}[label = (S.\arabic*)]
\item \label{s1} $\delta(u_i) =_\V \delta(v_i)$, for $i = 1, \ldots, n$;
\item \label{s3} $\delta(x) \in K_x$, for
  every variable $x \in X$.
\end{enumerate}
If $\delta(X) \subseteq \pseudosig AS$, then we say that
$\delta$ is a solution modulo $\V$ of \eqref{eq:35} \emph{in
  $\sigma$-words}.

\begin{rem}\label{sec:50}
  It follows from Hunter's Lemma that, for each clopen set $K_x$, there
  exists a finite semigroup $S_x$ and a continuous homomorphism $\varphi_x:
  \pseudo AS \to S_x$ such that $K_x$ is the preimage of
  $\varphi_x(K_x)$ under~$\varphi_x$ (see
  \cite[Proposition~3.5]{profinite}, for instance). It
  is sometimes more
  convenient to think of the constraints of the variables in terms of
  a fixed pair $(\varphi,\nu)$, where $\varphi: \pseudo AS \to S$ is a
  continuous homomorphism into a
  finite semigroup $S$ and $\nu: X \to S$ is a map.
  In that way, the requirement \ref{s3} becomes a finite union
  of requirements of the form 
  ``$\varphi(\delta(x)) = \nu_j(x)$, for every variable $x \in X$'',
  for a certain finite family $(\nu_j:X \to S)_{j} $ of mappings.
  We may also assume, without loss of generality that $S$ has a
  content function (see \cite[Proposition
  2.1]{MR1834943}).
  Moreover, usually, we  wish to allow $\delta$ to take its values
  in $(\pseudo AS)^I$.
  For that purpose, we naturally extend the function $\varphi$ to a 
  continuous homomorphism $\varphi^I: (\pseudo AS)^I \to S^I$ by letting
  $\varphi^I(I) = I$.
  It is worth noticing that this assumption does not lead to trivial
  solutions since the constraints must be satisfied.
  We allow ourselves some flexibility in these points, adopting each
  approach according to which is the most suitable. 
  In the case where we consider the homomorphism $\varphi^I$, we abuse
  notation and denote it by~$\varphi$.
\end{rem}

Given a class $\iC$ of finite systems of pseudoequations, one may pose
the following problem:
\begin{quote}
  determine whether a given system from $\iC$ (together with some
  constraints on variables)
  has a solution modulo $\V$.
\end{quote}
The pseudovariety $\V$ is \emph{$\iC$-decidable}
if the above decision problem is decidable.

An  important instance of a class of systems of equations comes from graphs. 
Let $\Gamma = V \uplus E$ be a directed graph, where $V$ and $E$ are
finite sets, respectively, of \emph{vertices} and \emph{edges}.
We consider $\Gamma$ equipped with two maps $\alpha: E \to V$ and
$\omega: E \to V$, such that an edge $e \in E$ goes from
the vertex $v_1 \in V$ to the vertex $v_2 \in V$ if and only if
$\alpha(e) = v_1$ and $\omega(e) = v_2$.
We may associate to each edge $e \in E$, the
equation $\alpha(e)e = \omega(e)$. We denote by $\iS(\Gamma)$
the finite system of equations obtained in this way from
$\Gamma$. Whenever $\iS$ is a finite system of this form, we say
that $\iS$ is a \emph{system of graph equations}.
We notice that any system of graph equations is of the form $\{x_i y_i
= z_i\}_{i = 1}^N$, where $y_i \neq y_j$ for $i \neq j$ and $y_i
\notin \{x_j, z_j\}$, for all $i,j$.
If $\iC$ is the class of all systems of graph equations, arising from a
graph with $n$ vertices at most,
then $\iC$-decidability deserves the name of
\emph{$n$-hyperdecidability} in
\cite{hyperdecidable}. The pseudovariety $\V$ is \emph{hyperdecidable} if it is
$n$-hyperdecidable for all $n \ge 1$.

When the constraints of the variables $e \in E$ are set to be given
by the clopen subset
$K_e = \{I\}$, the system $\iS(\Gamma)$ is called a
\emph{system of pointlike
  equations}. 
We say that $\V$ is \emph{strongly
  decidable}
if it is decidable for the class of all systems of
pointlike equations.

Next, we present some remarkable results involving these notions that
motivate us to consider the classes
of systems of (idempotent) pointlike and graph equations.

\begin{prop}
  [{\cite[Corollary 4]{hyperdecidable}}]
  \label{4hyperdecidable}
  Every strongly decidable pseudovariety is also decidable.
\end{prop}
\begin{thm}
  [{\cite[Theorem 14]{hyperdecidable}}]
  \label{14hyperdecidable}
  Let $n$ be a natural number, $\V$ a decidable pseudovariety of rank
  $n$ containing the Brandt semigroup $B_2$, and $\W$ a
  $(n+1)$-hyperdecidable pseudovariety. Then, $\V * \W$ is
  decidable.
\end{thm}
\begin{prop}
  [{\cite[Corollary 5]{silvaJA}}]
  \label{5silvaJA}
  If $\V$ is strongly decidable and $\W$ is
  order-computable, then $\V
  * \W$ is strongly decidable.
\end{prop}
\begin{thm}
  [{\cite[Theorem 15]{hyperdecidable}}]
  \label{15hyperdecidable}
  Let $\V$ be a hyperdecidable (respectively, strongly decidable)
  pseudovariety and let $\W$ be an order-computable
  pseudovariety. Then, $\V \vee \W$ is hyperdecidable (respectively,
  strongly decidable).
\end{thm}
\begin{thm}[{\cite[Theorem 4.1]{pinWEIL} and \cite[Theorem 4.2]{tameness}}]
  \label{t:1}
  If $\V$ is decidable and $\W$ is $\iC$-decidable for $\iC$
  consisting of systems of the form $x_1 = \cdots = x_n = x_n^2$, then
  $\V \malcev \W$ is decidable.
\end{thm}

We call systems of equations of the form
exhibited in Theorem \ref{t:1} \emph{systems of
  idempotent pointlike equations}.

However, since the semigroups $\pseudo AV$ are very often uncountable,
it is in general hard to say whether a pseudovariety $\V$ is
$\iC$-decidable, for a given class of systems $\iC$.
That was the motivation for the emergence of the next few concepts.

Given a class $\iC$ of finite systems of $\sigma$-equations, we say that a pseudovariety
$\V$ is \emph{$\sigma$-reducible with respect to $\iC$} (or simply,
\emph{$\sigma$-reducible for $\iC$})
provided a solution
modulo $\V$ of a system in $\iC$ guarantees the existence of a
solution modulo $\V$ of that system given by $\sigma$-words.
The pseudovariety $\V$ is said to be \emph{$\sigma$-reducible}
if it
is $\sigma$-reducible for the class of finite systems of graph
equations and it is
\emph{completely $\sigma$-reducible}
if it is $\sigma$-reducible for the class of all
finite systems of $\sigma$-equations.
The following result involves the notion of reducibility.
\begin{prop}
  [{\cite[Proposition 10.2]{unified_theory}}]
  \label{p:6}
  If $\V$ is $\sigma$-reducible with respect to the equation $x = y$,
  then $\V$ is $\sigma$-equational.
\end{prop}
Since we are aiming to achieve decidability results for $\V$, it is
reasonable 
to require that~$\V$ is recursively enumerable and that
$\sigma$ is \emph{highly computable},
meaning that it is
a recursively enumerable set and that all of its elements are computable
operations.
Henceforth, we make this assumption without further mention.
Also, we should be able to decide whether two given
$\sigma$-words have the same value over $\V$, the so-called
\emph{$\sigma$-word problem}.
When $\sigma = \kappa$ is the canonical implicit signature consisting
of the multiplication and of the $(\omega-1)$-power, it is possible to
characterize decidability of the $\kappa$-word problem for
pseudovarieties of the form $\DRH$ in terms of the same
property for $\h$.
\begin{thm}[{\cite[Chapter 3]{phd}}]\label{t:7}
  Let $\h$ be a pseudovariety of groups. Then, the pseudovariety
  $\DRH$ has decidable $\kappa$-word problem if and only if so has $\h$.
\end{thm}
We say that $\V$
is \emph{$\sigma$-tame with respect to $\iC$},
for a highly computable
implicit signature~$\sigma$, if it is $\sigma$-reducible for~$\iC$ and
has decidable $\sigma$-word problem.
We say that $\V$ is \emph{$\sigma$-tame}
(respectively, \emph{completely
  $\sigma$-tame}) when it is $\sigma$-tame with respect to the class of
finite systems of graph equations (respectively, to the class of all
finite systems of $\sigma$-equations).
\begin{thm}
  [{\cite[Theorem 10.3]{unified_theory}}]\label{t:5}
  Let $\iC$ be a recursively enumerable class of finite systems of
  $\sigma$-equations, without parameters. If $\V$ is a pseudovariety
  which is $\sigma$-tame with respect to $\iC$, then $\V$ is
  $\iC$-decidable.
\end{thm}
Despite being a stronger requirement, it is sometimes easier to prove
that a given pseudovariety is tame with respect to~$\iC$, rather than
its
$\iC$-decidability.

We end this subsection with a list of decidability results concerning
some pseudovarieties of groups, to which we refer later.

\begin{thm}\label{t:6}
  We have the following:
  \begin{itemize}
  \item the pseudovariety ${\sf Ab}$ is completely $\kappa$-tame
    (\cite{MR2142087});
  \item the pseudovariety ${\sf G}$ is $\kappa$-tame
    (\cite{MR1112302} and \cite[Theorem 4.9]{steinbergJA}),
    but it is not completely $\kappa$-reducible (\cite{MR1485465});
  \item for every extension closed pseudovariety of groups
    $\h$,
    there
    exists an implicit signature $\sigma(\h)$ such that $\h$ is
    $\sigma(\h)$-reducible (\cite{MR1859280});
  \item no proper subpseudovariety of $\G$ containing a
      pseudovariety $\G_p$ (for a
      certain prime~$p$) is $\kappa$-reducible (Proposition \ref{p:6}
      and \cite{MR0179239});
  \item no proper non locally finite subpseudovariety of ${\sf Ab}$ is
    $\kappa$-reducible (\cite{MR2364777}).
  \end{itemize}
\end{thm}

\section{Pointlike equations}\label{sec:1003}

Throughout this section, we shall assume that
$\sigma$ contains a non-explicit operation. In other words, that means
that $\langle\sigma\rangle \neq \langle\{\_\cdot\_\}\rangle$.
Clearly, that is the
case of the canonical implicit signature $\kappa$.

Propositions~\ref{4hyperdecidable} and~\ref{p:6}
motivate us to take for $\iC$ in Question~\eqref{q1} the class of all finite systems of
pointlike equations.
To guarantee that $\DRH$ is $\sigma$-reducible
for $\iC$, it suffices to suppose that $\h$ is $\sigma$-reducible
for~$\iC$ as well.

\begin{thm}\label{pointlike}
  Let $\sigma$ be an implicit signature
  containing a non-explicit operation, and
  assume that $\h$ is a pseudovariety of groups that is
  $\sigma$-reducible
  for finite systems of pointlike equations. Then, the pseudovariety
  $\DRH$ is also $\sigma$-reducible for finite systems of pointlike
  equations.
\end{thm}
\begin{proof}
  Let $\iS =  \{x_{k,1} = \cdots = x_{k, n_k}\}_{k = 1}^N$ be a finite
  system of pointlike equations in the set of variables~$X$ with
  constraints given by the pair $(\varphi,\nu)$.
  Without loss of generality, we may assume that, for all $k,\ell
  \in \{1, \ldots, N\}$, with $k \neq \ell$, the subsets of variables
  $\{x_{k,1}, \ldots, x_{k, n_k}\}$ and $\{x_{\ell,1}, \ldots, x_{\ell,
    n_\ell}\}$ do not intersect. Further, with this assumption, we may
  also take $N = 1$. The general case is obtained by treating each
  system of equations $x_{k,1} = \cdots = x_{k, n_k}$ separately.
  Write $\iS = \{x_1 = \cdots = x_n\}$ and
  suppose that the continuous homomorphism $\delta:
  \pseudo XS \to (\pseudo AS)^I$ is a solution modulo $\DRH$ of $\iS$. To prove
  that $\iS$ also has a solution in $\sigma$-words we argue by
  induction on $m = |c(\delta(x_{1}))|$.%
  
  If $m = 0$, then $\delta(x_{i}) = I$ for every $i= 1, \ldots, n$ and
  $\delta$ is already a solution in $\sigma$-words.
 
  Suppose that $m > 0$ and that the statement holds for every system
  of pointlike equations with a smaller value of the parameter.
  Whenever the $p$-th iteration of the left basic factorization of
  $\delta(x_i)$ is nonempty, we write $\lbf[p]{\delta(x_i)} =
  \delta(x_i)_pa_{i,p}$ and we let $\delta(x_i)_p'$ be such that
  $$\delta(x_i) = \lbf[1] {\delta(x_i)} \cdots \lbf[p]{\delta(x_i)} \delta(x_i)_p'.$$
  Notice that the uniqueness of left basic factorizations in
  $\pseudo A{DRH}$ entails the following properties
  \begin{equation}
    \begin{split}
      a_{1,p} & = \cdots =  a_{n,p};
      \\  \delta(x_{1})_p & =_\DRH \cdots =_\DRH \delta(x_{n})_p;
      \\ \delta(x_{1})_{p}' & =_\DRH \cdots =_\DRH \delta(x_{n})_{p}'.
    \end{split}\label{eq:20}
  \end{equation}
  If $\cum {\delta(x_1) }\neq c(\delta(x_1))$, then we set $\ell =
  \min\{p \ge 1\colon c(\delta(x_1)_p') \subsetneqq c(\delta(x_1))\}$.
  Otherwise, since $S$ is finite, there exist indices $k < \ell$ such
  that, for all $i = 1, \ldots, n$, we have
  \begin{align}
    \varphi(\lbf[1]{\delta(x_i)}\cdots \lbf[k]{\delta(x_i)}) =
    \varphi(\lbf[1]{\delta(x_i)}\cdots \lbf[\ell]{\delta(x_i)}).\label{eq:1}
  \end{align}
  Let $\eta \in \langle{\sigma}\rangle$ be a non-explicit
  operation. Without loss of generality, we may assume that $\eta$
  is a unary operation. In particular, since $S$ is finite, there
  is an
  integer $M$ such that $\varphi(\eta(s)) = s^{M}$ for every $s \in S$.
  Then, equality \eqref{eq:1} yields
  \begin{equation}
    \begin{split}
      \varphi(\delta(x_{i})) &= \varphi(\lbf[1]{\delta(x_i)} \cdots
      \lbf[k]{\delta(x_i)} \cdot \eta(\lbf[k+1]{\delta(x_i)} \cdots
      \lbf[\ell]{\delta(x_i)}) \delta(x_{i})_{k}').
    \end{split}\label{eq:21}
  \end{equation}
  Now, consider a new set of variables
  $X' = \{x_{i,p},x_i' \colon i = 1, \ldots, n; \: p = 1, \ldots, \ell \}$
  and a new system of pointlike equations
  \begin{equation}
    \iS' =
    \begin{cases}
      \{x_{1,p} = \cdots = x_{n, p}\}_{p= 1}^\ell
      \cup \{x_1' = \cdots = x_n'\}, \quad\text{ if $\cum {\delta(x_1) }\neq
        c(\delta(x_1))$}
      \\ \{x_{1,p} = \cdots = x_{n, p}\}_{p= 1}^\ell,
      \quad\text{ if  $\cum{\delta(x_1)} = c(\delta(x_1))$}
    \end{cases}\label{eq:43}
  \end{equation}
  By \eqref{eq:20}, the continuous homomorphism
  $\delta': \pseudo {X'}S \to (\pseudo AS)^I$ assigning $\delta(x_i)_p$
  to each variable $x_{i,p}$ and $\delta(x_i)_\ell'$ to each variable
  $x_i'$ is a solution modulo $\DRH$ of $\iS'$,
  with constraints given by $(\varphi,\nu')$, where $\nu'(x_{i,p}) =
  \varphi(\delta(x_i)_p)$, and  $\nu'(x_{i}') =
  \varphi(\delta(x_i)_k')$ ($i = 1, \ldots, n$ and $p = 1, \ldots, \ell$).
  Moreover, whatever is the system $\iS'$ considered in \eqref{eq:43},
  we decreased the induction parameter.
  By induction hypothesis, there exists a solution modulo
  $\DRH$ of $\iS'$ in $\sigma$-words, say $\varepsilon'$, keeping the
  values of the variables under $\varphi$.
  We distinguish between the case where $\cum {\delta(x_1) }\neq
  c(\delta(x_1))$ and the case where  $\cum{\delta(x_1)} = c(\delta(x_1))$. In
  the former, it is easy to check that the continuous homomorphism
  \begin{align*}
    \varepsilon: \pseudo XS &\to (\pseudo AS)^I
    \\ x_{i} &\mapsto \varepsilon'(x_{i,1})a_{i,1} \cdots
               \varepsilon'(x_{i,\ell})a_{i,\ell} \varepsilon'(x_i')
  \end{align*}
  is a solution modulo $\DRH$ of $\iS$.
  In the latter case, we consider the system of pointlike equations $$\iS_0 =
  \{x_1' = \cdots = x_n'\}.$$
  From \eqref{eq:20}, it follows that
  $\delta'$ is a solution modulo $\h$ of $\iS_0$. As we are taking for
  $\h$ a pseudovariety that is $\sigma$-reducible for systems of
  pointlike equations,
  there exists a solution modulo $\h$ of $\iS_0$, say
  $\varepsilon''$, keeping
  the values of the variables under $\varphi$.
  Let $\varepsilon: \pseudo XS \to (\pseudo AS)^I$ be given by
  $$\varepsilon(x_{i}) = \varepsilon'(x_{i,1})a_{i,1} \cdots
  \varepsilon'(x_{i,k})a_{i,k} \cdot\eta(\varepsilon'(x_{i,k+1})a_{i,k+1}
  \cdots \varepsilon'(x_{i,\ell})a_{i,\ell})
  \varepsilon''(x_i').$$
  Since $\varepsilon'$ is a
  solution modulo $\DRH$ of $\iS'$, $\eta$ is non-explicit, and
  we are assuming that the semigroup $S$ has a content function, it
  follows that, for all $i,j \in \{1, \ldots , n\}$, the
  pseudowords~$\varepsilon(x_i)$
  and~$\varepsilon(x_j)$ are $\Req$-equivalent modulo~$\DRH$.
  On the other hand, for all $i, j \in \{1, \ldots, n\}$, the
  following equalities are valid in~$\h$:
  \begin{align*}
    \varepsilon(x_{i})
    & = \varepsilon'(x_{i,1})a_{i,1} \cdots
      \varepsilon'(x_{i,k})a_{i,k} \cdot \eta(\varepsilon'(x_{i,k+1})a_{i,k+1}
      \cdots \varepsilon'(x_{i,\ell})a_{i,\ell})
      \varepsilon''(x_i')
    \\ & = \varepsilon'(x_{j,1})a_{j,1} \cdots
         \varepsilon'(x_{j,k})a_{j,k} \cdot \eta(\varepsilon'(x_{j,k+1})a_{j,k+1}
         \cdots \varepsilon'(x_{j,\ell})a_{j,\ell})
         \varepsilon''(x_j')
    \\ &= \varepsilon(x_{j}) .
  \end{align*}
  The second equality holds because $\varepsilon'$ and
  $\varepsilon''$ are solutions modulo $\h$ of $\iS'$ and $\iS_0$, respectively.
  Therefore, Lemma \ref{sec:16} yields that $\DRH$ satisfies
  $\varepsilon(x_i) = \varepsilon(x_j)$. It remains to verify that the
  given constraints are still satisfied. But that is
  straightforwardly implied by~\eqref{eq:21}.
\end{proof}
\begin{rem}
  \label{sec:17}
  We observe that the construction performed in the proof of the previous
  theorem not only gives a solution modulo $\DRH$ in
  $\sigma$-terms of the original pointlike system of equations, but it
  also provides a solution keeping the cumulative content of each
  variable.
\end{rem}

As a consequence of Propositions \ref{p:6} and Theorem~\ref{pointlike}, we have the following.

\begin{cor}\label{c:12}
  If a pseudovariety of groups $\h$ is $\sigma$-reducible with respect
  to the equation $x = y$, then $\DRH$ is $\sigma$-equational.\qed
\end{cor}
As far as we are aware, all known examples of pseudovarieties of groups that
are $\sigma$-reducible with respect to systems of pointlike equations
are also $\sigma$-reducible.
For that reason, for now, we skip such examples, since
they illustrate stronger results in the next section.
We just point out the case of the pseudovariety ${\sf Ab}$ (recall
Theorem \ref{t:6}).
It is interesting to observe that, although $\overline \Ab$ is not a
$\kappa$-equational pseudovariety \cite[Theorem~3.1]{MR1980404}, by
Corollary \ref{c:12} the
pseudovariety ${\sf DRAb} = {\sf DRG}\cap \overline \Ab$~is.

On the other hand, taking into account Theorem~\ref{t:7}, we also have
the following.

\begin{cor}\label{c:13}
  If $\h$ is a pseudovariety of groups that is $\kappa$-tame with
  respect to finite systems of pointlike equations, then so is
  $\DRH$.\qed
\end{cor}

Since, by Theorem \ref{t:5}, $\kappa$-tame pseudovarieties are
hyperdecidable (with respect to a certain class $\iC$), another
application comes from Proposition~\ref{5silvaJA} and Theorem \ref{15hyperdecidable}.
\begin{cor}
  Let $\h$ be a pseudovariety of groups that is $\kappa$-tame with
  respect to systems of pointlike equations and $\V$ an order
  computable pseudovariety. Then, both $\DRH * \V$ and
  $\DRH \vee \V$ are strongly decidable pseudovarieties.
  \qed
\end{cor}
\section{Graph equations}\label{sec:1004}
With the aim of proving tameness, we now let $\iC$ be the class
of all systems
of graph equations. Results on tameness of $\DRH$ also allow us to
know more about pseudovarieties of the form $\V * \DRH$ and $\DRH \vee
\V$ for certain pseudovarieties $\V$ (recall
Theorems~\ref{14hyperdecidable} and~\ref{15hyperdecidable}).
We prove
that, for an implicit signature $\sigma$ containing a non-explicit
operation,
if $\h$ is a $\sigma$-reducible pseudovariety of groups, then so is
$\DRH$.
To this end, we drew inspiration from~\cite{Rtame}.
Moreover, we assert the converse statement, which holds for every~$\sigma$.

Henceforth, we fix a finite graph $\Gamma = V  \uplus E$ and a solution 
$\delta: \pseudo {\Gamma}S \to (\pseudo AS)^I$ modulo~$\DRH$
of $\iS(\Gamma)$ such that, for every $x \in \Gamma$ the pseudoword
$\delta(x)$ belongs to the clopen subset $K_x$ of $(\pseudo AS)^I$.
We further denote by~$1$ the identity element of $\pseudo AH$.

Let $y$ be an edge of $\Gamma$, and let $x = \alpha(y)$ and $z = \omega(y)$.
If
$c(\delta(y)) \nsubseteq \cum{\delta(x)}$ then, by Lemma~\ref{c:5}, we
have unique factorizations $\delta(y) = u_yav_y$ and
$\delta(z) = u_zav_z$ such that $c(u_y) \subseteq
\cum{\delta(x)}$, $a \notin \cum{\delta(x)}$ and the pseudovariety
$\DRH$ satisfies both $\delta(x)u_y = u_z$ and $v_y =v_z$.
We refer to these factorizations as \emph{direct
  $\DRH$-splittings associated with the edge~$y$} and we say that $a$
is the corresponding \emph{marker}. We call \emph{direct
  $\DRH$-splitting 
  points}
the triples $(u_y, a,v_y)$ and $(u_z, a,v_z)$.

The first remark spells out the relationship between the notion of a
$\DRH$-splitting factorization defined above and the notion of a
splitting factorization in the context of~\cite{Rtame}
(in~\cite{Rtame}, a splitting factorization is defined as being an
$\R$-splitting factorization). It is a consequence of Lemma~\ref{c:5}
applied to the pseudovariety $\DRH$ and
to the pseudovariety $\R$.

\begin{rem}\label{sec:5}
  Let $y \in E$ be such that
  $c(\delta(y))
  \nsubseteq \cum{\delta(\alpha(y))}$. Consider factorizations $\delta(y) =
    u_y a v_y$ and $\delta(\omega(y)) = u_z a v_z$, with $c(u_y)
  \subseteq \cum{\delta(\alpha(y))}$ and $a \notin
  \cum{\delta(\alpha(y))}$, such that $\DRH$ satisfies
  $\delta(\alpha(y))u_y = u_z$, as
  above.
  Then, these factorizations
  are direct $\R$-splittings (note that $\delta$ is also a
  solution modulo $\R$ of $\iS(\Gamma)$ and so, it makes sense to
  refer to $\R$-splitting factorizations) if and only if they are direct
  $\DRH$-splittings.
\end{rem}

We also define the
\emph{indirect $\DRH$-splitting points} as follows.
Let $t \in \Gamma$ and suppose that we have a factorization $\delta(t)
= u_tav_t$, with $a \notin \cum{u_t}$. Then, one of the following
three situations may occur.
\begin{itemize}
\item If there is an edge $y \in E$ such that $\alpha(y) = t$ and
  $\omega(y) = z$, then there is also a factorization $\delta(z) =
  u_zav_z$ with $\DRH$ satisfying $u_t = u_z$ and $v_t\delta(y) =
  v_z$.
  In fact, this is a consequence of the pseudoidentity $\delta(t)
  \delta(y) = \delta(z)$ modulo $\DRH$, which holds for every edge
  $t \xrightarrow y z$ in $\Gamma$.
\item Similarly, if there is an edge $y \in E$ such that $\alpha(y) =
  x$ and $\omega(y) = t$ (and so, $\DRH$ satisfies $\delta(x)\delta(y)
  = \delta(t)$), then the factorization of $\delta(t)$ yields either a factorization
  $\delta(x) = u_xav_x$ such that $\DRH$ satisfies $u_x = u_t$ and
  $v_x\delta(y) = v_t$,
  or a factorization $\delta(y) = u_yav_y$
  such that $\DRH$ satisfies $\delta(x) u_y = u_t$ and $v_y =  v_t$.
\item On the other hand, if $t$ is itself an edge, say $\alpha(t) = x$
  and $\omega(t) = z$, and if $\delta(x)u_ta$ is an end-marked
  pseudoword, then the factorization of $\delta(t)$ determines a
  factorization $\delta(z) = u_zav_z$, such that $\DRH$ satisfies
  $\delta(x)u_z = u_t$ and $v_z = v_t$.
\end{itemize}
These considerations make clear the possible propagation of the
$\DRH$-direct 
splitting points. If the mentioned factorization of $\delta(t)$ comes
from a $\DRH$-(in)direct splitting factorization obtained through the
successive factorization of the values of edges and vertices under
$\delta$ in the way described above, then we say that each of the
triples $(u_x, a,x_x)$, $(u_y, a, v_y)$, and $(u_z, a, v_z)$ is an
indirect $\DRH$-splitting point
\emph{induced by the (in)direct
  $\DRH$-splitting point $(u_t, a, v_t)$}.
In Figure \ref{fig:3} we schematize a propagation of
splitting points arising from the direct $\DRH$-splitting point associated
with the edge $y_1$. We represent pseudowords by boxes, markers of splitting points by
dashed lines and factors with the same value modulo $\DRH$ with the
same filling pattern.
\begin{figure}[htpb]
  \centering
  \begin{tikzpicture}
    \footnotesize
    \draw [->] (0,0) arc (135:45:60pt);
    \node at (-22pt,-10pt) {$\delta(x_1) = $};
    \draw (-5pt,-15pt) rectangle (5pt,-5pt);
    \node at (18pt,25pt) {$\delta(y_1) = $};
    \draw (35pt, 21pt) rectangle (50pt,31pt);
    \draw[dashed,very thick] (40pt,33pt) -- (40pt,19pt);
    \node at (58pt,-10pt) {$\delta(x_2) = $};
    \draw (75pt,-15pt) rectangle (95pt,-5pt);
    \draw[dashed,very thick] (85pt,-17pt) -- (85pt,-3pt);
    \draw [->] (85pt,-20pt) arc (-135:-45:60pt);
    \node at (138pt,-10pt) {$\delta(x_3) = $};
    \draw (155pt,-15pt) rectangle (190pt,-5pt);
    \draw[dashed,very thick] (165pt,-17pt) -- (165pt,-3pt);
    \node at (118pt,-45pt) {$\delta(y_2) = $};
    \draw (135pt,-50pt) rectangle (150pt,-40pt);
    \draw [<-] (175pt,0pt) arc (135:45:60pt);
    \node at (248pt,-10pt) {$\delta(x_4) = $};
    \draw (265pt,-15pt) rectangle (275pt,-5pt);
    \node at (198pt,25pt) {$\delta(y_3) = $};
    \draw (215pt, 21pt) rectangle (240pt,31pt);
    \draw[dashed,very thick] (225pt,33pt) -- (225pt,19pt);
    \fill[pattern=north west lines, pattern color=gray] (40pt,21pt)
    rectangle (50pt,31pt);
    \fill[pattern=north west lines, pattern color=gray] (85pt,-15pt)
    rectangle (95pt,-5pt);
    \fill[pattern=crosshatch dots, pattern color=gray] (75pt,-15pt)
    rectangle (85pt,-5pt);
    \fill[pattern=crosshatch dots, pattern color=gray] (155pt,-15pt)
    rectangle (165pt,-5pt);
    \fill[pattern=crosshatch dots, pattern color=gray] (215pt,21pt) rectangle (225pt,31pt);
  \end{tikzpicture}
  \caption{Example of propagation of a direct splitting point.}
  \label{fig:3}
\end{figure}
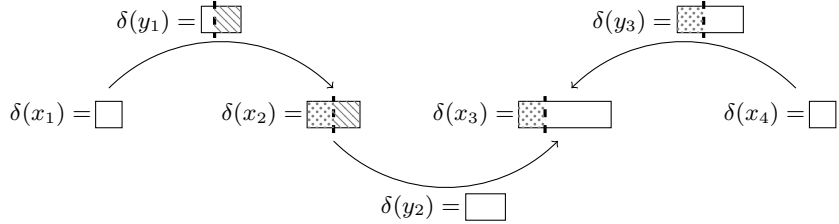

Yet again, we obtain a nice relationship between the indirect
$\DRH$-splitting points just defined and the indirect splitting
points introduced in \cite{Rtame} (which are the
indirect $\R$-splitting points). The reason is precisely the same as
in Remark \ref{sec:5}, together with the definition of indirect
splitting points.

\begin{rem}\label{sec:4}
  Let $t_0 \in \Gamma$ and $\delta(t_0) = u_0av_0$ be a direct
  $\R$-splitting factorization and consider $\{(u_i, a, v_i)\}_{i =
    1}^n \subseteq (\pseudo AS)^I \times A \times (\pseudo AS)^I$. Then,
  the following are equivalent:
  \begin{enumerate}
  \item $(u_i, a, v_i)$ is an indirect $\R$-splitting point induced
    by $(u_{i-1}, a, v_{i-1})$, for $i = 1, \ldots, n$;
  \item $(u_i, a, v_i)$ is an indirect $\DRH$-splitting point induced
    by $(u_{i-1}, a, v_{i-1})$, for $i = 1, \ldots, n$.
  \end{enumerate}
\end{rem}

The following lemma ensures that a direct $\R$-splitting point does
not propagate infinitely many times.
\begin{lem}
  [{\cite[Lemma 5.14]{Rtame}}] \label{sec:6}
  Given a solution $\delta$ over $\R$ of a system of graph equations,
  there is only a finite number of splitting points in the values of
  variables under $\delta$.
\end{lem}
As an immediate consequence of Lemma \ref{sec:6} and of
the relationship between (in)direct $\R$-splitting points and
(in)direct $\DRH$-splitting points made explicit
in Remarks~\ref{sec:5} and~\ref{sec:4} we have the following:
\begin{cor}
  Given a solution $\delta$ over $\DRH$ of a system of graph equations,
  there is only a finite number of splitting points in the values of
  variables under $\delta$. \qed
\end{cor}

Taking into account Remarks \ref{sec:5} and \ref{sec:4}, from now on
we say
(in)direct splitting point (respectively, factorization) instead of
(in)direct $\DRH$-splitting point (respectively, factorization).

Let $\Gamma$ be a finite graph and consider the system of
equations $\iS(\Gamma)$.
For each
variable $x \in \Gamma$, let $\{(u_{x,i}, a_{x,i}, v_{x,i})\}_{i =
  1}^{m_x}$ be the (finite) set of splitting points of $\delta(x)$. By
definition, each pseudoword $u_{x,i}a_{x,i}$ is an end-marked prefix
of $\delta(x)$. By Proposition \ref{p:2}, we may assume,
without loss of generality, the following relations:
$$u_{x,1}a_{x,1} >_{\Req} u_{x,2}a_{x,2} >_{\Req} \cdots >_{\Req}
  u_{x, m_x}a_{x,m_x} >_{\Req}\delta(x).$$
Hence, by Lemma~\ref{c:5}, we have a reduced factorization
\begin{equation}
  \begin{split}
    \delta(x) &= \delta(x)_0
    \cdot \delta(x)_1
    \cdots \delta(x)_{m_x},
  \end{split}
\label{eq:17}
\end{equation}
such that $\delta(x)_0\cdots \delta(x)_{i-1} = u_{x,i}$, for $i = 1,
\ldots, m_x$,
induced by the splitting points of~$\delta(x)$. If $x
\in V$, then we write the reduced factorization in \eqref{eq:17} as
$\delta(x) = w_{x,1}\cdot w_{x,2} \cdots w_{x, n_x}$ and, if $y \in
E$, then we write that factorization as
$\delta(y) = w_{y,0} w_{y,1} \cdots w_{y,
  n_{y}}$.
Observe that, for $x \in V$, we have $n_{x} = m_x+1$, while for $y \in E$, we
have $n_{y} = m_{y}$.
Although this notation may not seem coherent, it is
justified by property \ref{item:10} of Lemma \ref{sec:23}.

\begin{lem}
  \label{sec:23}
  Let $xy = z$  be an equation of $\iS(\Gamma)$.
  Using the above notation, the following holds:
  \begin{enumerate}
  \item \label{item:8} $n_{x}+n_{y} = n_{z}$;
  \item \label{item:9}$\DRH$ satisfies $
    \begin{cases}
      w_{x, k} = w_{z,k}, \quad\text{ for }k = 1, \ldots, n_{x}-1;
      \\ w_{x, n_{x}} w_{y, 0} = w_{z,n_{x}};
      \\ w_{y, k} = w_{z, n_{x} + k}, \quad\text{ for }k = 1,
      \ldots, n_{y};
    \end{cases}$
  \item \label{item:10} $c(w_{y,0})\subseteq \cum{w_{x,
        n_{x}}}$;
  \item \label{item:11} each of the following products is reduced:
    \begin{align*}
      &w_{x, k} \cdot w_{x,
        k+1} \: (k = 1, \ldots, n_{x}-1);
      \\ & (w_{x, n_{x}} w_{y, 0})\cdot
           w_{y, 1};
      \\ & w_{z, k} \cdot w_{z, k+1} \: {(k = 1, \ldots,
      n_{z}-1)}.
    \end{align*}
  \end{enumerate}
\end{lem}
\begin{proof}
  As we already observed, the number of splitting
  points of $\delta(z)$ is $m_{z} = n_{z} - 1$.
  We distinguish between two situations.
  \begin{itemize}
  \item If $c(\delta(y)) \nsubseteq \cum{\delta(x)}$, then
    there are two direct splitting factorizations given by
    $\delta(y) = u_{y}av_{y}$ and $\delta(z) =
    u_{z}av_{z}$.
    So, by definition, the inclusion $c(u_{y}) \subseteq
    \cum{\delta(x)}$ holds.
    We notice that any other splitting point of $\delta(y)$, say
    $(u_{y}', b, v_y')$, is necessarily induced by a splitting point of
    $\delta(z)$, say $(u_z', b, v_z')$.
    Moreover, since the product $(\delta(x)u_y') \cdot b v_y'$ is
    reduced (because so is~$u_z' \cdot( b v_z')$ and~$\DRH$
      satisfies $\delta(x)u_y' = u_z'$), the
    pseudoword $u_y$ is a prefix of $u_y'$.
    On the other hand, the set of all splitting points of
    $\delta(z)$ induces a factorization of the pseudoword
    $\delta(x)\delta(y)$, say
    \begin{equation}
        \delta(x)\delta(y) = w_1'\cdot w_2'\cdots w_{n_{z}}'.\label{eq:37}
    \end{equation}
    Of course, for each $k = 1, \ldots, n_{x}-1$, the prefix
    $w_1'\cdots w_k'$ of $\delta(x)\delta(y)$
    corresponds to  the first 
    component of one of the splitting points of $\delta(x)$ (which is either induced
    by one of the splitting points of $\delta(z)$ or it induces a
    splitting point of $\delta(z)$).
    More specifically, the pseudoidentity
    $w_{z,k} = w_k' = w_{x,k}$
    is valid in $\DRH$.
    From the observation above, we also know that the first components
    of the indirect splitting points of $\delta(y)$ have $u_y$ as a
    prefix. Therefore, we have  $u_y = w_{y,0}$, the factor
    $w_{n_{x}}' = w_{z, n_{x}}$
    coincides with $w_{x,n_{x}}w_{y,0}$ modulo $\DRH$, and
    $c(w_{y,0}) = c(u_y) \subseteq
    \cum{\delta(x)}=\cum{w_{x,n_{x}}}$. It also follows that
    $w_{n_{x}+k}' = w_{z, n_{x}+k} = w_{y, k}$ modulo $\DRH$,
    for $k = 1, \ldots, n_{y}$. We just proved \ref{item:9},
    \ref{item:10} and \ref{item:11}. Finally, part \ref{item:8} results
    from counting the involved factors in both sides of \eqref{eq:37}.
  \item If $c(\delta(y)) \subseteq \cum{\delta(x)}$, then
    $\delta(y)$ has no direct splitting points. As $y$ is an edge,
    an indirect splitting point of $\delta(y)$ must be induced by
    some splitting point of $\delta(z)$. Suppose that $(u_z, a,
    v_z)$ is a splitting point of $\delta(z)$ that induces a
    splitting point in $\delta(y)$, say $(u_y, a, v_y)$. Then, we
    would have a reduced product $(\delta(x)u_y) \cdot (a v_y)$, which
    contradicts the assumption $c(\delta(y)) \subseteq
    \cum{\delta(x)}$. Therefore, the pseudoword $\delta(y)$ has no
    splitting points at all. With the same kind of argument as the one
    above, we may derive the claims \ref{item:8}--\ref{item:11}. \popQED
  \end{itemize}
\end{proof}

Now, write $\iS(\Gamma) =  \{x_iy_i = z_i\}_{i =
  1}^N$. Note that $y_j \notin \{x_i, z_i\}$ for all $i,j$.
We let $\iS_1$ be the system of equations containing, for each $i
= 1, \ldots, N$, the following set of equations:
\begin{equation}
  \begin{split}
    (x_i)_k & = (z_i)_k, \text{ for } k = 1, \ldots, n_{x_i}-1;
    \\ (x_i)_{n_{x_i}} y_{i,0} & = (z_i)_{n_{x_i}};
    \\ y_{i,k}  &= (z_i)_{n_{x_i} + k}, \text{ for } k = 1, \ldots, n_{z_i}.
  \end{split}\label{eq:30}
\end{equation}
In the system $\iS_1$, we are assuming that $(x_i)_k$
and $(x_j)_k$ (respectively, and $(z_j)_k$)  represent the same
  variable whenever so do $x_i$  and $x_j$ (respectively, and $z_j$).
By Lemma~\ref{sec:23}, it is clear that each solution modulo $\DRH$ of
$\iS_1$ yields a solution modulo $\DRH$ of $\iS(\Gamma)$ and
conversely. We next prove that, for a $\sigma$-reducible
pseudovariety of groups $\h$, if $\iS_1$ has a
solution modulo $\DRH$, then it has a solution modulo $\DRH$ given by
$\sigma$-words, thus concluding that the same happens
with~$\iS(\Gamma)$. Before that, we establish the following. 
\begin{prop}\label{sec:10}
  Let $\sigma$ be an implicit signature that contains a non-explicit
  operation.
  Let $\h$ be a $\sigma$-reducible pseudovariety of groups and
  $\Gamma = V \uplus E$ be a finite graph. Suppose that 
  there exists a solution $\delta: \pseudo {\Gamma}S \to (\pseudo
  AS)^I$ modulo $\DRH$ of $\iS(\Gamma)$ such that:
  \begin{enumerate}
  \item $\cum{\delta(x)} \neq \emptyset$, for every vertex $x \in V$
  \item\label{item:1} $c(\delta(y)) \subseteq \cum{\delta(\alpha(y))}$, for
    every edge $y \in E$.
  \end{enumerate}
  Then, $\iS(\Gamma)$ has a solution modulo $\DRH$ in $\sigma$-words, say
  $\varepsilon$, such that $\varphi(\varepsilon(x)) =
  \varphi(\delta(x))$, for all $x \in\Gamma$.
\end{prop}
\begin{proof}
  Without loss of generality, we may assume that $\Gamma$ has only one
  connected component (when disregarding the directions of the
  arrows). Otherwise, we may treat each component separately.
  Because of the hypothesis \ref{item:1}, the pseudowords
  $\delta(\alpha(y))$ and 
  $\delta(\omega(y))$ are $\Req$-equivalent modulo $\DRH$ for every
  edge $y \in E$.
  Since we are assuming that all vertices of $\Gamma$ are in the
  same connected component, it follows that for all $x,z \in V$,
  the pseudowords $\delta(x)$ and $\delta(z)$ are $\Req$-equivalent
  modulo $\DRH$.
  Fix a variable $x_0 \in V$ and let
  $u_0$ be an accumulation
  point of $(\lbf[1]{\delta(x_0)} \cdots \lbf[m]{\delta(x_0)})_{m \ge 1}$ in~$\pseudo AS$. Since,
  in $\DRH$, the pseudowords $u_0$ and $\delta(x_0)$ are
  $\Req$-equivalent, for each $x \in V$ there is a factorization
  $\delta(x) = u_x v_x$ (with $v_x$ possibly empty) such that $c(v_x)\subseteq \cum{u_x}$ and $u_x
  =_\DRH u_0$.
  
  Consider the set $\widehat V = \{\widehat x \colon x \in V\}$ with $\card V$
  distinct variables, disjoint from  $\Gamma$, the system of
  equations $\iS_0 = \{\widehat
  x = \widehat z \colon x, z \in V\}$ with variables in $\widehat V$, and let
  \begin{align*}
    \delta_0: \pseudo {\widehat V \uplus \Gamma} S &\to (\pseudo AS)^I
    \\ \widehat x & \mapsto u_x,  \quad\text{ if } \widehat x \in \widehat V;
    \\ x & \mapsto v_x, \quad\text{ if } x \in V;
    \\ y & \mapsto \delta(y), \quad\text{ otherwise}.
  \end{align*}
  By construction, the homomorphism $\delta_0$ is a solution modulo
  $\DRH$ of $\iS_0$ which is also a solution modulo $\h$ of
  $\iS(\Gamma)$.
  Hence, on the one hand, Theorem~\ref{pointlike}
  together with Remark~\ref{sec:17} yield a solution
  $\varepsilon_0:\pseudo{\widehat V} S \to \pseudo AS$ modulo $\DRH$
  in $\sigma$-words of $\iS_0$ such that
  \begin{align*}
    \varphi(\varepsilon_0(\widehat x)) & = \varphi(\delta_0(\widehat x))
                                         = \varphi(u_x),
    \\ \cum{\varepsilon_0(\widehat x)} & = \cum{\delta_0(\widehat x)},
  \end{align*}
  for every $\widehat x \in \widehat V$. On the other hand, the fact
  that $\h$ is $\sigma$-reducible implies that there is a solution
  $\varepsilon': \pseudo\Gamma S \to (\pseudo AS)^I$ modulo $\h$ of
  $\iS(\Gamma)$ given by $\sigma$-words satisfying
  $$\varphi(\varepsilon'(x)) =
  \varphi(\delta_0(x)) = \varphi(v_x),$$ for every $x \in \Gamma$.
  Thus, we take $\varepsilon: \pseudo {\Gamma}S \to (\pseudo AS)^I$
  to be
  the continuous homomorphism defined by $\varepsilon(x) =
  \varepsilon_0 (\widehat x)\varepsilon'(x)$ if $x \in V$, and
  $\varepsilon(y) = \varepsilon'(y)$ otherwise. Taking into account that $S$ has a
  content function, we may use Lemma~\ref{sec:16} to deduce that
  $\varepsilon$ is a solution modulo $\DRH$ of $\iS(\Gamma)$ in
  $\sigma$-words. It is easy to check that the constraints for the
  variables of $\Gamma$ are also satisfied. Therefore, $\varepsilon$
  is the required homomorphism.
\end{proof}
\begin{lem}
  Let $\iS_1$ be the system of equations~\eqref{eq:30} in the set of
  variables $X_1$ and let $\delta_1: \pseudo
  {X_1}S \to (\pseudo AS)^I$ be its solution modulo $\DRH$.
  Suppose that the implicit signature $\sigma$ contains a non-explicit operation.
  If the pseudovariety~$\h$ is
  $\sigma$-reducible, then $\iS_1$ has
  a solution modulo $\DRH$ in $\sigma$-words.
\end{lem}
\begin{proof}
  Analyzing the equations in \eqref{eq:30},
  we easily conclude that there are no variables
  occurring simultaneously in the first row and in one of the other two
  rows. Therefore, the system $\iS_1$ can be thought as a system of
  pointlike equations $\iS_2$ together with a system of graph
  equations $\iS_3$ such that the condition \ref{item:1} of
  Proposition~\ref{sec:10} holds and none of the variables occurring
  in $\iS_2$ occurs
  in $\iS_3$. Note that we are also including in $\iS_2$ the equations
  in the last two rows of \eqref{eq:30} such that the cumulative
  content of the value of the involved variables under $\delta_1$ is
  empty.
  
  By Theorem~\ref{pointlike} the system $\iS_2$ has a solution modulo $\DRH$ in
  $\sigma$-words, while by Proposition~\ref{sec:10} the system $\iS_3$ has a
  solution modulo $\DRH$ in $\sigma$-words. Therefore, the intended
  solution for $\iS_1$ also exists.
\end{proof}
We just proved the announced result.
\begin{thm}\label{DRH_graph_eq}
  When $\sigma$ is an implicit signature containing a non-explicit
  signature, the pseudovariety $\DRH$ is $\sigma$-reducible if so is
  $\h$.
  \qed
\end{thm}

We recall that, by Theorem \ref{t:6}, for every nontrivial
extension closed pseudovariety of groups~$\h$, there is an implicit
signature $\sigma(\h) \supseteq \kappa$ that turns ${\h}$
into a $\sigma(\h)$-reducible pseudovariety.
For instance,~$\G_p$ and~$\G_{sol}$
are both extension closed. Thus, ${\sf DRG}_p$ and ${\sf DRG}_{sol}$
are both $\sigma$-reducible for suitable signatures~$\sigma$.

Yet again, using Theorem~\ref{DRH_graph_eq}, some decidability
properties may be deduced from the
knowledge of $\kappa$-tameness of a pseudovariety of groups $\h$.

\begin{cor}
  Let $\h$ be a $\kappa$-tame pseudovariety of groups. Then,
  \begin{itemize}
  \item $\DRH$ is $\kappa$-tame (Theorem~\ref{t:7});
  \item $\V * \DRH$ is decidable for every decidable pseudovariety
    $\V$ containing the Brandt semigroup~$B_2$ (Theorems~\ref{t:5} and~\ref{14hyperdecidable});
  \item $\V \vee \DRH$ is hyperdecidable for every order computable
    pseudovariety $\V$ (Theorem~\ref{15hyperdecidable}).\qed
  \end{itemize}
\end{cor}

We further prove that the converse of Theorem~\ref{DRH_graph_eq} also holds.

\begin{prop}\label{l:4}
  Let $\h$ be a pseudovariety of groups such that the pseudovariety
  $\DRH$ is $\sigma$-reducible. Then,
  the pseudovariety $\h$ is also $\sigma$-reducible.
\end{prop}
\begin{proof}
  Let $\Gamma = V \uplus E$ be a graph such that $\iS(\Gamma)$
  admits $\delta: \pseudo \Gamma S
  \to (\pseudo AS)^I$ as a solution modulo $\h$. We consider a new graph
  $\widehat \Gamma = \widehat V \uplus \widehat E$, where $\widehat V
  = \{\widehat v \colon v \in V\} \uplus \{v_0\} $ and $\widehat E = V
  \uplus E $. The functions $\alpha$ and $\omega$ of $\widehat \Gamma$
  are given by $\alpha(v) = v_0$ and $\omega(v) = \widehat v$, for all
  $v \in V$ and by $\alpha(e) = \widehat v_1$ and $\omega(e) =
  \widehat v_2$ whenever $e \in E$ and $(\alpha(e), \omega(e)) =
  (v_1, v_2)$.
  The relationship between the graphs $\Gamma$ and $\widehat \Gamma$
  is depicted in Figure~\ref{fig:13}.
  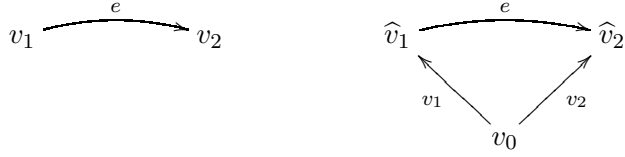
\begin{figure}[htpb]
    \centering
    $\xymatrix{v_1 \ar@/^/[rr]^e && v_2 && \widehat v_1 \ar@/^/[rr]^e
      &&\widehat  v_2
      \\ &&&&&v_0 \ar [ul]^{v_1}\ar [ur]_{v_2}&}$
    \caption{On the left, an edge of $\Gamma$; on the right, the
      corresponding edges of $\widehat\Gamma$.}
    \label{fig:13}
  \end{figure}
  Let $u \in \pseudo AS$ be a regular pseudoword modulo $\DRH$ such
  that $c(\delta(x)) \subseteq \cum u$ for all~$x \in \Gamma$.
  We take $\delta': \pseudo {\widehat \Gamma} S \to (\pseudo AS)^I$ to
  be the continuous homomorphism defined by $\delta'(e) = \delta(e)$,
  if~$e \in E$; $\delta'(v) = u\delta(v)$ and $\delta'(\widehat v) = u
  \delta(v)$, if $v \in V$; and $\delta'(v_0) = u$.
  Then, Lemma \ref{sec:16} combined with the fact that $\delta$ is a
  solution modulo $\h$ of $\iS(\Gamma)$ imply that $\delta'$ is a
  solution modulo $\DRH$ of $\iS(\widehat \Gamma)$. Thus, since $\DRH$
  is $\sigma$-reducible, there exists a
  solution in $\sigma$-words $\varepsilon: \pseudo{\widehat \Gamma}S
  \to (\pseudo AS)^I$ modulo $\DRH$ of $\iS(\widehat\Gamma)$.
  In particular, for each edge $e \in E$ such that $\alpha(e) = v_1$
  and $\omega(e) = v_2$, we have that
  $v_0v_1= \widehat v_1 $, $\widehat v_1e =
  \widehat v_2$, and $v_0v_2= \widehat v_2 $ are equations of~$\iS(\widehat \Gamma)$.
  Therefore, the equalities $\varepsilon(v_0v_1e) =
  \varepsilon(\widehat v_1e) = \varepsilon(\widehat v_2) =\varepsilon(v_0v_2)$
  hold in $\DRH$.
  Hence, $\h$ satisfies $\varepsilon(v_1e) =
  \varepsilon(v_2)$ and so, we may conclude that the restriction of
  $\varepsilon$ to $\pseudo \Gamma 
  S$ is a solution in $\sigma$-words modulo $\h$ of $\iS(\Gamma)$.
\end{proof}

Combined with Proposition~\ref{l:4}, the results in the literature supply a
family of pseudovarieties
$\DRH$ that are not $\kappa$-reducible. Namely, ${\sf DRG}_p$ and
$\DRH$ for every proper non locally finite 
subpseudovariety~$\h$ of~${\sf Ab}$ (recall Theorem \ref{t:6}).
\section{Idempotent pointlike equations}\label{sec:1005}

Theorem \ref{t:1} provides a sufficient criterion for decidability
of pseudovarieties of the form $\V \malcev \DRH$, whenever $\V$ is a
decidable pseudovariety. With that fact in mind, we take for~$\iC$
the class of all systems of idempotent pointlike equations.
In the preceding two situations, the answers to Question~\eqref{q1}
were of the form ``it is enough to assume that~$\h$ is $\sigma$-reducible with
respect to $\iC$''.
When considering systems of idempotent pointlike equations,
we have been unable to give such an answer.
However, we prove that assuming $\sigma$-reducibility of $\h$ with
respect to a still ``satisfactory'' class of systems serves our
purpose.
More precisely, we prove that, for an implicit signature $\sigma$
satisfying certain conditions, if $\h$ is a $\sigma$-reducible
pseudovariety of groups, then $\DRH$ is $\sigma$-reducible with
respect to systems of idempotent pointlike equations.

In order to make the expression ``reducible for systems of graph
equations'' more embracing, we first introduce a definition.
\begin{deff}
  Let $\V$ be a pseudovariety and $\iS$  a finite system of equations
  in the set of variables $X$ with certain
  constraints. We say that
  $\iS$ is
  \emph{$\V$-equivalent to a system of graph equations} if
  there exists a  graph  $\Gamma$  such that $X \subseteq \Gamma$
  and such that every solution modulo~$\V$ of~$\iS$ can be extended to
  a solution modulo $\V$ of $\iS(\Gamma)$
  (the constraints for the variables of $X \subseteq \Gamma$ are those
  given by the system $\iS$).
  Moreover, whenever $\delta$ is a solution modulo~$\V$
  of~$\iS(\Gamma)$, the restriction $\delta|_{\pseudo {X} S}$ is a
  solution
  modulo~$\V$ of~$\iS$.
  Each graph $\Gamma$ with that property is said to be an \emph{$\iS$-graph}
  and we say that $\iS$ is \emph{$\V$-equivalent to $\iS(\Gamma)$}
  for every $\iS$-graph $\Gamma$.
\end{deff}
It is immediate from the definition that any $\sigma$-reducible
pseudovariety $\V$
is $\sigma$-reducible for systems of equations that are $\V$-equivalent
to a system of graph equations.
In the next few results we exhibit some systems of equations that are
$\h$-equivalent to a system of graph equations (for a pseudovariety of
groups~$\h$).
Instead of giving complete proofs, we identify on each situation
what graph should be considered and leave the details to the reader.

\begin{lem}
  \label{sec:2}
  Let $\iS = \{ x_1 w_1 \cdots
  x_n w_n x_{n+1} = 1\}$ be a system consisting of a single equation,
  where $x_i$ is a variable with $x_i \neq x_j$ whenever $i \neq j$,
  $\{w_i\}_{i = 1}^n \subseteq A^*$, and
  the constraint of the
  variable $x_i$ is given by the clopen subset $K_i \subseteq (\pseudo
  AS)^I$. Then, for every pseudovariety of groups~$\h$, the
  system~$\iS$ is $\h$-equivalent to a system of graph equations.
\end{lem}
\begin{proof}
  Let $\Gamma = V \uplus E$ be the finite graph with the sets of
  vertices and edges given by $V = \{y_i, z_i \colon i = 1,
  \ldots, n+1 \}$ and $E  =  \{x_0\} \uplus \{ x_i \colon i = 1,
    \ldots, 
    n+1\} \uplus \{ w_i \colon i = 1, \ldots, n\}$, respectively.
  To define the mappings~$\alpha$ and~$\omega$, we take
  \begin{align*}
    (\alpha(x_0), \omega(x_0)) &= (y_1, z_{n+1});
    \\ (\alpha(x_i), \omega(x_i)) & = (y_i, z_i), \quad \text{ for $i = 1, 
                                    \ldots, n+1$};
    \\ (\alpha(w_i), \omega(w_i)) & = (z_i, y_{i+1}), \quad \text{ for
                                    $i = 1,
                                    \ldots, n$ };
  \end{align*}
  as shown in Figure~\ref{fig:1}.
   \begin{figure}[htpb]\centering
    $\xymatrix{
      y_1 \ar @/^/  [r]^{x_1}  \ar@/_20pt/ [rrrrrrr]_{x_0}&
      z_1 \ar @/^/  [r]^{w_1} &
      y_2 \ar @/^/  [r]^{x_2} &
      z_2 \ar @/^/  [r]^{w_2} & \ar@{}[r] |{\cdots}& z_n
      \ar @/^/  [r]^{w_n} &
      y_{n+1} \ar @/^/ [r]^{x_{n+1}} &z_{n+1}}$
    \caption{The graph $\Gamma$.}
    \label{fig:1}
  \end{figure}
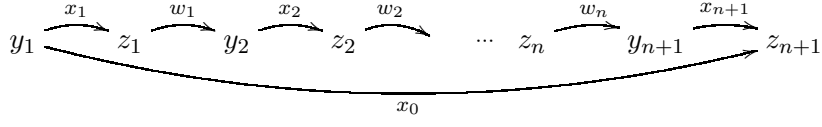
  Let us the denote by $K_x$ the clopen set that defines the
  constraint of~$x \in \Gamma$. We set
  $K_{x_i} = K_i$, $K_{w_i} = \{w_i\}$, $K_{x_0} = \{I\}$, and
  $K_{y_{i}} = (\pseudo AS)^I = K_{z_i}$ (for every $i$ such that each
  of the variables is defined).
  Then, $\Gamma$ is an $\iS$-graph.
\end{proof}

\begin{lem}
  \label{sec:3}
  Let $\h$ be a pseudovariety of groups.
  If $\iS$ is $\h$-equivalent to a
  system of graph equations, $x$ is
  a variable occurring in $\iS$, and $\iS_0 = \{x = x_1 = \cdots =
  x_n\}$, where $x_1, \ldots, x_n$ are new variables, then $\iS
  \cup\iS_0$ is also $\h$-equivalent to a system of graph equations.
\end{lem}
\begin{proof}
  Let $\Gamma = V \uplus E$ be an $\iS$-graph. We construct a new
  graph $\Gamma'$ as follows.
  If $x \in V$, then we consider new variables $x_0, x_1, \ldots, x_n$
  and $\Gamma'$ is obtained by adding to $\Gamma$ the edges
  represented in Figure~\ref{fig:15} on the left.
  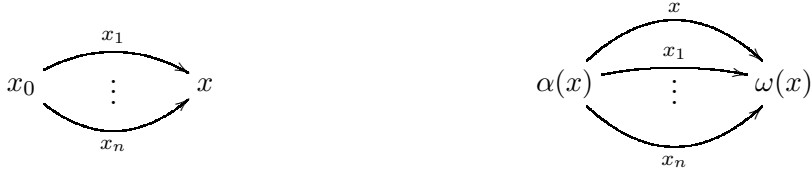
\begin{figure}[htpb]
    \centering
    $\xymatrix{ x_0 \ar@/^1pc/[rr]^{x_1}\ar@/_1.5pc/[rr]_{x_n}
      \ar@{}[rr] |{\vdots}&& x & &&&\alpha(x)\ar@/^2pc/[rr]^{x} \ar@/^/[rr]^{x_1}\ar@/_2pc/[rr]_{x_n}
      \ar@{}[rr] |{\vdots} && \omega(x)}$
    \caption {The piece of the graph $\Gamma'$ where it differs from
      $\Gamma$, when $x \in V$ (left) and when~$x \in E$ (right).}
    \label{fig:15}
  \end{figure}
  Otherwise, if $x \in E$, then $n$ new edges are added to~$\Gamma$  as
  depicted in Figure~\ref{fig:15} on the right, resulting the graph $\Gamma'$.
  We do not explicit the constraints on the new variables, since they
  may be taken to be given by the clopen set~$\pseudo AS$.
  In any case, it is a routine matter to verify that the system $\iS\cup\iS_0$ is
  $\h$-equivalent to~$\iS(\Gamma')$.
\end{proof}

\begin{lem}\label{l:5}
  Let $\h$ be a pseudovariety of groups,
  $\iS$ be a system of equations with variables in $X$ that is
  $\h$-equivalent to a system of  graph equations and $\iS_0 = \{x =
  x_1 w_1 \cdots x_n w_n 
  x_{n+1}\}$, where $x \in X$, $x_1,
  \ldots, x_{n}$ are new variables, $x_{n+1}$ is either the empty word
  or a new variable, and $\{w_i\}_{i = 1}^n \subseteq A^*$. Then, $\iS \cup
  \iS_0$ is also $\h$-equivalent to a system of graph equations.
\end{lem}
\begin{proof}
  We start by observing that it really does not matter whether
  $x_{n+1}$ is the empty word or a new variable. Indeed, if it is the
  empty word, then we just need to set a constraint $K_{x_{n+1}} = \{I\}$ for
  $x_{n+1}$ and we may treat it as a variable.

  Let $\Gamma = V \uplus E$ be an $\iS$-graph. We construct a new
  graph $\Gamma'$ depending on whether~$x$ is a
  vertex or an edge.
  If $x$ is a vertex, then we add to $\Gamma$ a new path going from
  a new vertex $y_1$ to $x$, whose edges are labeled by $x_1, w_1,
  \ldots, x_n, w_n,x_{n+1}$ in this order, as depicted in Figure
  \ref{fig:14}.
  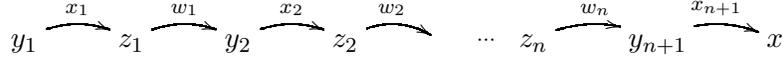
\begin{figure}[htpb]
    \centering
    $\xymatrix{
      y_1 \ar @/^/  [r]^{x_1} &
      z_1 \ar @/^/  [r]^{w_1} &
      y_2 \ar @/^/  [r]^{x_2} &
      z_2 \ar @/^/  [r]^{w_2} & \ar@{}[r] |{\cdots}& z_n
      \ar @/^/  [r]^{w_n} &
      y_{n+1} \ar @/^/ [r]^{x_{n+1}} &x}$
    \caption{The new path in $\Gamma$ if $x$ is a vertex.}
    \label{fig:14}
  \end{figure}
  We further take $K_{y_1} = \{I\}$, $K_{y_{i+1}} = (\pseudo AS)^I =
  K_{z_i}$, and $K_{w_i} = \{w_i\}$ as the clopen
  sets defining the constraints for the new variables $y_{i+1}$, $z_i$,
  and $w_i$, respectively ($i = 1, \ldots, n$).
  On the other hand, when $x$ is an edge, we
  simply obtain $\Gamma'$ by adding a path in $\Gamma$ from $\alpha(x)$ to
  $\omega(x)$ with edges labeled by $x_1, w_1, \ldots, x_n, w_n,
  x_{n+1}$ (see Figure~\ref{fig:2}).
  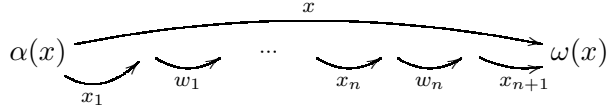
\begin{figure}[htpb]
    \centering
    $\xymatrix{
      \alpha(x) \ar @/_1pc/  [r]_{x_1}  \ar@/^1pc/ [rrrrrr]^{x}&
       \ar @/_0.5pc/  [r]_{w_1} &\ar@{}[r] |{\cdots}&
       \ar @/_0.5pc/  [r]_{x_n} &
       \ar @/_0.5pc/  [r]_{w_n} &
      \ar @/_0.5pc/  [r]_{x_{n+1}} &\omega(x)}$
    \caption{The added path to $\Gamma$ if $x$ is an
      edge.}
    \label{fig:2}
  \end{figure}
  We leave it to the reader to verify that, in both
  situations, $\iS\cup \iS_0$ is $\h$-equivalent to $\iS(\Gamma')$.
\end{proof}
\begin{cor}
  \label{c:1}
  Let $\h$ be a pseudovariety of groups and
  let $\iS$ be a system of equations $\h$-equivalent to a system of
  graph equations and suppose that $x_1, \ldots, x_N$ are variables
  occurring in $\iS$. Also, suppose that the
  variables $y_{i,1}, \ldots, y_{i,k_i}$ and $z_{i,1}, \ldots, z_{i,
    n_i}$ ($i = 1, \ldots, N$) are all distinct and do not occur in $\iS$, and let
  $\{w_{i,p}\colon i = 1, \ldots, N; \: p = 1, \ldots, k_i\} \subseteq
  A^*$. We make each $t_i$ be either the empty word or another
  different variable.
  Then, the system of equations
  \begin{align*}
  \iS' = \iS  & \cup \{x_i = y_{i,1} w_{i,1} \cdots y_{i, k_i}
  w_{i,k_i} t_i\}_{i = 1}^N \\ &\cup \{ t_i = z_{i,1} = \cdots = z_{i,
    n_i} \colon i = 1, \ldots, N \text{ and } t_i \text{ is not the
    empty word}\}
  \end{align*}
  is also $\h$-equivalent to a system of graph equations.
\end{cor}
\begin{proof}
  The result follows immediately by successively applying Lemmas \ref{sec:3}
  and \ref{l:5}.
\end{proof}

The next statement consists of a general scenario that is instrumental
for establishing the claimed answer to Question \eqref{q1} mentioned in
the beginning of this section.

\begin{prop}\label{eq_especiais}
  Let $\h$ be a $\sigma$-reducible pseudovariety
  of groups, where $\sigma$ is an implicit signature such that
  $\langle\sigma\rangle$ contains a non-explicit operation $\eta$ such
  that  $\eta = 1$ in $\h$.
  Let $\iS_1$ and $\iS_2$ be finite systems of equations, such that
  $\iS_1$ contains only pointlike equations, and both $\iS_1 \cup \iS_2$ and
  $\iS_2$ are $\h$-equivalent to systems of
  graph equations.
  Further assume that, if $X$ is the set of variables occurring in
  $\iS_1 \cup\iS_2$, then the constraint on each variable $x \in X$ is given by
  a clopen subset $K_x \subseteq (\pseudo AS)^I$.
  Then, the existence of a continuous homomorphism that is
  simultaneously a solution modulo~$\DRH$ 
  of~$\iS_1$ and a solution modulo~$\h$ of~$\iS_2$ entails the
  existence of
  a continuous homomorphism in $\sigma$-words with the same property.
\end{prop}
\begin{proof}
  Without loss of generality, we assume that $\eta$ is a unary implicit operation.
  Let $\iS_1 = \{x_{i,1} = \cdots = x_{i,n_i}\}_{i = 1}^N$, with
  $x_{i,p} \neq x_{j,q}$, for all $i \neq j$.
  We consider a continuous
  homomorphism $\varphi: (\pseudo AS)^I \to S^I$ such that each clopen
  set $K_x$ is the preimage of a
  finite subset of $S^I$ under~$\varphi$ (recall Remark~\ref{sec:50}).
  We argue by induction on the parameter $$M = \max\{\card{c(\delta(x_{i,p}))}\colon i =
  1, \ldots, N, \: p = 1, \ldots, n_i\}.$$

  If $M = 0$, then $\delta(x_{i,p}) = I$ for all $i = 1, \ldots, N$
  and $p = 1, \ldots, n_i$ and therefore,
  every solution $\varepsilon$ modulo $\h$ of $\iS_2$ such that
  $\varepsilon (x_{i,p}) = I$ (for $i = 1, \ldots, N$, and $p = 1,
  \ldots, n_i$) is trivially a solution modulo $\DRH$ of $\iS_1$.
  Since we are assuming that $\iS_2$ is $\h$-equivalent to a system of graph
  equations and we are taking for $\h$ a $\sigma$-reducible
  pseudovariety, there exists such an $\varepsilon$ given by $\sigma$-words.
  
  Suppose that $M > 0$ and that the result holds for any smaller
  parameter.
  If $\delta(x_{i,p})$ has empty cumulative content, then
  we let $k_i$ be the maximum integer such that
  $\lbf[k_i]{\delta(x_{i,p})}$ is nonempty and we write $\lbf[m]{\delta(x_{i,p})} = \delta(x_{i,p})_ma_{i,m}$, for $m =
  1, \ldots, k_i$.
  Otherwise, for each $m \ge
  1$, we consider the $m$-th iteration of the left basic
  factorization to the right of $\delta(x_{i,p})$, say $\delta(x_{i,p})
    = \delta(x_{i,p})_1a_{i,1} \cdots \delta(x_{i,p})_{m} a_{i,m}
    \delta(x_{i,p})_m'$.
  Since $S$ and $A$ are both finite, there are integers $1 \le k < \ell$ such
  that, 
  for all $i$, $p$ satisfying $\cum{\delta(x_{i,p})} \neq \emptyset$, we
  have
  \begin{align*}
    \vec{c}(\delta(x_{i,p}))  &= c(\delta(x_{i,p})_{k+1}a_{i,k+1});
    \\  \varphi(\delta(x_{i,p})) &= \varphi( \delta(x_{i,p})_1a_{i,1}
                                   \cdots \delta(x_{i,p})_{k} a_{i,k})
    \\ & \quad\cdot \eta( \delta(x_{i,p})_{k+1}a_{i,k+1}
         \cdots\delta(x_{i,p})_{\ell} a_{i,\ell})
         \delta(x_{i,p})_k').
  \end{align*}

  Now, consider a new set of variables $X'$ given by the union
  \begin{align*}
    X \uplus \{x_{i,p;m}, a_{i,m} \colon i = 1, \ldots, N; \: p =
             1, \ldots, n_i; \:
                        \vec{c}(\delta(x_{i,p})) = \emptyset; \,m = 1,
                        \ldots, k_i\}
    \\ \uplus \{x_{i,p;m}, a_{i,m}, x_{i,p}' \colon i = 1, \ldots, N; \: p =
             1, \ldots, n_i; \: \vec{c}(\delta(x_{i,p})) \neq
         \emptyset;\,m = 1, \ldots, \ell\},
  \end{align*}
  where all the introduced variables are distinct.
  In order to simplify the notation, we set $\ell_i = 0$ if
  $\cum{\delta(x_{i,p})} = \emptyset$, and $k_i = k$ and $\ell_i = \ell$,
  otherwise.
  We further take the constraints on~$X'$ to be given by $K_x$ if $x \in X$,
  and by the clopen sets $K_{x_{i,p;m}} =
  \varphi^{-1}(\varphi(\delta(x_{i,p})_m)$,
  $K_{a_{i,m}} = \{a_{i,m}\}$, and $K_{x_{i,p}'} =
  \varphi^{-1}(\varphi(\delta(x_{i,p})'_k)$ for the remaining cases.
  
  Consider the system
  $$\iS_1' = \{x_{i,1;m} = \cdots = x_{i,n_i;m} \colon i = 1,  \ldots,
  N; \: m = 1, \ldots, \max\{k_i,\ell_i\}\}.$$
  A new system $\iS_2'$ is
  obtained from the system $\iS_1 \cup \iS_2$ (which is
  $\h$-equivalent to a system of graph equations, by hypothesis) by
  adding two sets of equations:
  \begin{itemize}
  \item for each $i = 1, \ldots, N$, if there exists an index $p$ such
    that $x_{i,p}$ is a variable occurring in $\iS_2$, then we choose
    such an index, say $p_i$. Then, we add the equation $$x_{i,p_i} =
    x_{i,p_i;1}a_{i,1} \cdots x_{i,p_i;k_i}a_{i,k_i} z_{i,p_i},$$
    where $z_{i,p_i}$ stands for the empty word if $ \ell_i = 0$, and for $x_{i,p_i} '$ otherwise;
  \item and we add the set of equations $$\{x_{i,1}' =
    \cdots = x_{i,n_i}' \colon i = 1, \ldots, N, \:\ell_i \neq 0\}.$$
  \end{itemize}
  By Corollary~\ref{c:1}, the new system $\iS_2'$ is still $\h$-equivalent
  to a system of graph equations.
  Moreover, if we denote by $X_j'$ the set of variables occurring in $\iS_j'$ ($j
  = 1,2$), then the following equality holds:
  $$X_1' \cap X_2' = \{x_{i,p_i;m} \colon i = 1, \ldots, N;
  \text{ $p_i$ is defined; and } m = 1, \ldots, \max\{k_i,\ell_i\}\}.$$
  Thus, again Corollary \ref{c:1} yields that the system $\iS_1'\cup
  \iS_2'$ is  $\h$-equivalent to a system of graph
  equations as well.
  Let $\delta': \pseudo {X'}S \to
  (\pseudo AS)^I$ be the continuous homomorphism defined by
  \begin{align*}
    \delta'(x_{i,p;m})
    & = \delta(x_{i,p})_m,\quad\text{if $i = 1,
      \ldots, N$; $p = 1, \ldots,n_i $; $m = 1,
      \ldots, \max\{k_i, \ell_i\}$};
    \\ \delta'(x_{i,p}')
    & = \delta(x_{i,p})_k' ,\quad \text{ if $i = 1,
      \ldots, N$; $p = 1, \ldots,n_i $;}
    \\ \delta'(x) & = \delta(x), \quad\text{ otherwise.}
  \end{align*}
  It follows from its definition that
  $\delta'$ is a solution modulo $\DRH$ of $\iS_1'$ which is also a
  solution modulo~$\h$ of~$\iS_2'$.
  Since the induction parameter corresponding to the triple $(\iS_1',
  \iS_2', \delta')$ is smaller than the one corresponding to the triple
  $(\iS_1, \iS_2, \delta)$, we may use the induction hypothesis to deduce
  the existence of a continuous homomorphism $\varepsilon':\pseudo {X'}
  S \to (\pseudo AS)^I$ in $\sigma$-words that is both a solution modulo $\DRH$ of
  $\iS_1'$ and a solution modulo $\h$ of $\iS_2'$.

  Finally, we define $\varepsilon$ as follows:
  \begin{align*}
    \varepsilon : \pseudo XS & \to(\pseudo AS)^I
    \\ x_{i,p} & \mapsto \varepsilon'(x_{i,p;1})a_{i,1} \cdots
    \varepsilon'(x_{i,p;k_i}) a_{i,k_i}, \quad \text{ if }
    \ell_i=0;
    \\ x_{i,p} & \mapsto \varepsilon'(x_{i,p;1})a_{i,1} \cdots
    \varepsilon'(x_{i,p;k}) a_{i,k}
    \\ & \quad \cdot \eta(\varepsilon'(x_{i,p;k+1})a_{i,k+1} \cdots
      \varepsilon'(x_{i,p;\ell})a_{i,\ell})\cdot
    \varepsilon'(x_{i,p}'), \quad\text{ if } \ell_i \neq 0;
    \\ x & \mapsto  \varepsilon'(x), \quad\text{ otherwise.}
  \end{align*}
  Then, a straightforward computation shows that $\varepsilon$ plays the
  desired role.
\end{proof}
    
We now state and prove the result claimed at the beginning of the
section.
\begin{thm}\label{sec:11}
  Let $\sigma$ be an implicit operation such that there exists
    $\eta \in \langle \sigma\rangle$ non-explicit, with $\eta = 1$ in $\h$.
  If $\h$ is a $\sigma$-reducible pseudovariety of groups, then $\DRH$ is
  $\sigma$-reducible for idempotent pointlike systems of equations.
\end{thm}
\begin{proof}
  Let $\iS = \{x_1 = \cdots = x_n = x_n^2\}$ be an idempotent
  pointlike system of equations with constraints on the variables
  given by the pair $(\varphi, \nu)$, and let $\delta: \pseudo
  {\{x_1, \ldots, x_n\}} S \to \pseudo AS$ be a solution modulo $\DRH$
  of $\iS$. Suppose that $\delta(x_i) = u_i$.
  Since idempotents over $\DRH$ are precisely the pseudowords with
  cumulative content coinciding with the content and with value $1$ over~$\h$
  (cf.\ \cite[Corollary~6.1.5]{drh}), $\DRH$ satisfies
  $u_1 = \cdots = u_n = u_n^2 $
  if and only if the following conditions hold:
  \begin{align}\label{eq:51}
    c(u_n)   &= \vec{c}(u_n);\nonumber
    \\   u_n & =_\h 1;
    \\  u_1  &=_\DRH \cdots =_\DRH u_n.\nonumber
  \end{align}
  For each $i \in \{1, \ldots, n\}$ and $m \ge 1$, let $u_i =
  \lbf[1]{u_i}\cdots \lbf[m]{u_i} u_{i,m}'$ and $\lbf[m]{u_i} = u_{i,m}a_m$. Since $S$ is finite,
  there are positive integers $k < \ell$ such that for all $i = 1,
  \ldots, n$ the equality $$\varphi(\lbf[1]{u_i} \cdots \lbf[k]{u_i}) =
  \varphi(\lbf[1]{u_i} \cdots \lbf[\ell]{u_i})$$ holds.
  Take the set of variables $$X = \{x_{i,p} \colon i = 1, \ldots,
  n; \: p = 1, \ldots, \ell\} \uplus \{x_i' \colon i = 1, \ldots, n
  \},$$
  with constraints given by $(\varphi,\nu')$, where $\nu'(x)
  =\varphi(u_{i,p}) $ if $x  = x_{i,p}$, and $\nu'(x) = \varphi(u_{i,k}')$
  if $x = x_i'$.
  We consider the systems of equations  $\iS_1 = \{x_{1,p} = \cdots = x_{n,p}\}_{p = 1}^\ell$ and $\iS_2 =
  \{x_{n,1}a_1 \cdots x_{n,k}a_k x_n' = 1, \: x_1' = \cdots = x_n'\}$.
  Then, the homomorphism
  \begin{align*}
    \delta': \pseudo XS &\to (\pseudo AS)^I
    \\ x_{i,p} &\mapsto u_{i,p}, \quad\text{ for } i = 1, \ldots, n;\: p = 1,
    \ldots, \ell;
    \\ x_i' &\mapsto u_{i,k}', \quad\text{ for } i = 1, \ldots, n;
  \end{align*}
  is a solution modulo $\DRH$ of $\iS_1$ that is also a solution
  modulo $\h$ of $\iS_2$. Besides that, since by Lemma~\ref{sec:2} the
  system $\{x_{n,1}a_1 \cdots x_{n,k}a_k x_n' = 1\}$ is
  $\h$-equivalent to a system of graph equations, Lemma~\ref{sec:3}
  yields that so is $\iS_2$.
  In turn, again Lemma~\ref{sec:3} implies that $\iS_1 \cup \iS_2$ is
  $\h$-equivalent to a system of graph equations.
  Thus, we may invoke Proposition
  \ref{eq_especiais} to derive the existence of a continuous
  homomorphism $\varepsilon': \pseudo XS \to (\pseudo AS)^I$ in
  $\sigma$-words that is a solution modulo $\DRH$ of $\iS_1$, and a
  solution modulo $\h$ of $\iS_2$.

  Now, assuming that $\eta$ is unary, we let
  $\varepsilon: \pseudo {\{x_1, \ldots, x_n\}} S \to \pseudo
  AS$ be given by
  $$\varepsilon(x_i) = \varepsilon'(x_{i,1}) a_1 \cdots
             \varepsilon'(x_{i,k})a_k \cdot \eta(\varepsilon'(x_{i,k+1})
             a_{k+1} \cdots \varepsilon'(x_{i,\ell})a_\ell)
             \varepsilon'(x_i').$$
  It is easily checked that $\DRH$ satisfies $ \varepsilon(x_1) =
  \cdots = \varepsilon(x_n)$, and $\h$ satisfies $\varepsilon(x_n) = 1$.
  Furthermore, by the choice of $k$ and $\ell$, we also know that
  $\varphi(\varepsilon(x_i)) = \varphi(\delta(x_i))$ and, as we are
  assuming that $\eta$ is non-explicit and $S$ has a content function, the equality $\vec{c}(\varepsilon(x_i))
  = c(\varepsilon(x_i))$ holds. So, by \eqref{eq:51}, we may conclude
  that $\varepsilon$ is a
  solution modulo $\DRH$ of $\iS$ in $\sigma$-words that keeps the
  values under~$\varphi$. 
\end{proof}
We observe that, whenever the $\omega$-power belongs to
$\langle\sigma\rangle$, the hypothesis of Theorem~\ref{sec:11} concerning
the implicit signature $\sigma$ is satisfied. That is the case of the
canonical implicit signature~$\kappa$.
Hence, we have the following.
\begin{cor}
  Let $\h$ be a $\kappa$-tame pseudovariety of groups. Then,
  \begin{itemize}
  \item $\DRH$ is $\kappa$-tame with respect to finite systems of
    idempotent pointlike equations (Theorem~\ref{t:7});
  \item $\V \malcev \DRH$ is decidable whenever $\V$ is a decidable
    pseudovariety (Theorem \ref{t:1}).\qed
  \end{itemize}
\end{cor}

In particular, the pseudovarieties ${\sf DRG}$ and ${\sf DRAb}$ are
both $\kappa$-tame with respect to finite systems of idempotent
pointlike equations and ${\sf DRG}_p$ and ${\sf DRG}_{sol}$ are
$\sigma$-reducible with respect to the same class (recall Theorem
\ref{t:6}).
\newline
\paragraph{\bf Acknowledgments.}  This work is part of the author's
Ph.D. thesis, written under supervision of Professor Jorge Almeida, to
whom the author is deeply grateful.
The work was partially supported by CMUP
(UID/MAT/ 00144/2013), which is funded by FCT (Portugal) with national
(MEC) and European structural funds through the programs FEDER, under
the partnership agreement PT2020 and by  FCT doctoral scholarship
SFRH/BD/ 75977/2011, with national (MEC) and European
structural funds through the program POCH.

\def\cprime{$'$}
\providecommand{\bysame}{\leavevmode\hbox to3em{\hrulefill}\thinspace}
\providecommand{\MR}{\relax\ifhmode\unskip\space\fi MR }
\providecommand{\MRhref}[2]{%
  \href{http://www.ams.org/mathscinet-getitem?mr=#1}{#2}
}
\providecommand{\href}[2]{#2}

\end{document}